\documentclass{amsart}
\usepackage{amssymb, amsmath, a4wide}
\usepackage[all]{xy} 
\usepackage{paralist} 
\usepackage{graphicx} 
\usepackage[sans]{dsfont} 
\usepackage{varioref} 
\usepackage{subfigure}
\usepackage{tikz}
\usepackage{pgf}
\usetikzlibrary{decorations.pathreplacing,calc}
\usepackage{xifthen}
\usepackage{hyperref}

\title{Analytic Compactifications of $\cc^2$ part I - curvettes at infinity}
\author{Pinaki Mondal}


\DeclareMathOperator{\sign}{sign}
\makeatletter

\newcommand{\Rmnum}[1]{\expandafter\@slowromancap\romannumeral #1@}
\makeatother

\DeclareMathOperator\ord{ord}

\DeclareMathOperator\sing{Sing} 
\DeclareMathOperator\spec{Spec}

\newcommand{\scrE}{\ensuremath{\mathcal{E}}}

\newcommand{\scrI}{\ensuremath{\mathcal{I}}}

\newcommand{\scrK}{\ensuremath{\mathcal{K}}}

\newcommand{\scrM}{\ensuremath{\mathcal{M}}}

\newcommand{\cc}{\ensuremath{\mathbb{C}}}

\newcommand{\nn}{\ensuremath{\mathbb{N}}}
\newcommand{\pp}{\ensuremath{\mathbb{P}}}
\newcommand{\qq}{\ensuremath{\mathbb{Q}}}

\newcommand{\zz}{\ensuremath{\mathbb{Z}}}

\newcommand{\sheaf}{\ensuremath{\mathcal{O}}}

\newcommand{\im}{\ensuremath{\Rightarrow}}

\newcommand{\into}{\ensuremath{\hookrightarrow}}

\newtheorem{thm}{Theorem}[section]
\newtheorem*{thm*}{Theorem}
\newtheorem{lemma}[thm]{Lemma}
\newtheorem*{lemma*}{Lemma}

\newtheorem{prop}[thm]{Proposition}
\newtheorem*{prop*}{Proposition}
\newtheorem{cor}[thm]{Corollary}

\newtheorem*{claim*}{Claim}

\newtheorem*{conjecture*}{Conjecture}

\theoremstyle{definition} 
\newtheorem{algorithm}[thm]{Algorithm}

\newtheorem*{constrinition*}{Construction-Definition}

\newtheorem*{convention*}{Convention}
\newtheorem{defn}[thm]{Definition}
\newtheorem*{defn*}{Definition}

\newtheorem*{definotation*}{Definition-Notation}
\newtheorem{example}[thm]{Example}
\newtheorem*{example*}{Example}

\newtheorem*{fact*}{Fact}

\newtheorem*{facts*}{Facts}

\newtheorem{notation}[thm]{Notation}

\newtheorem*{bold-note*}{Note}

\newtheorem*{problem*}{Problem}
\newtheorem{bold-question}[thm]{Question}
\newtheorem*{bold-question*}{Question}
\newtheorem{rem}[thm]{Remark}
\newtheorem{reminition}[thm]{Remark-Definition}
\newtheorem*{reminition*}{Remark-Definition}

\newtheorem{remexample*}{Remark-Example}
\newtheorem{remtation}[thm]{Remark-Notation}
\newtheorem*{remtation*}{Remark-Notation}

\newtheorem*{remuestion*}{Remark-Question}

\theoremstyle{remark}
\newtheorem*{rem*}{Remark}
\newtheorem*{note*}{Note}
\newtheorem*{notation*}{Notation}
\newtheorem*{question*}{Question}

\newtheorem*{questions*}{Questions}

\newcounter{UnorderedProofTempCtr}
\newcommand{\tempcommand}{}

\newcommand{\hot}{\text{h.o.t.}}
\newcommand{\lot}{\text{l.o.t.}}

\newcommand{\sumption}{\eqref{assumption}}
\newcommand{\sigmaxy}{\sigma_{(x,y)}}
\newcommand{\Xxy}{\bar X_{(x,y)}}

\begin{document}
\maketitle

\begin{abstract} 
We study normal analytic compactifications of $\cc^2$ and describe their singularities and configuration of curves at infinity, in particular improving and generalizing results of \cite{brentonfication}. As a by product we give new proofs of Jung's theorem on polynomial automorphisms of $\cc^2$ and Remmert and Van de Ven's result that $\pp^2$ is the only smooth analytic compactification of $\cc^2$ for which the curve at infinity is irreducible. We also give a complete answer to the question of existence of compactifications of $\cc^2$ with prescribed divisorial valuations at infinity. In particular, we show that a valuation on $\cc(x,y)$ centered at infinity determines a compactification of $\cc^2$ iff it is {\em positively skewed} in the sense of \cite{favsson-tree}. \\

Nous \'etudions les compactifications analytiques normales de $\cc^2$ et d\'ecrire leurs singularit\'es et la configuration des courbes \`a l'infini, particulierment am\'elioratant et g\'en\'eralisant les r\'esultats de \cite{brentonfication}. Comme un sous-produit, nous donnons de nouvelles preuves du th\'eor\`eme de Jung sur automorphismes polynomiaux de $\cc^2 $ et le r\'esultat de Remmert et Van de Ven que $\pp^2$ est le seul compactifi\'e analytique lisse $\cc^2$ pour lequel la courbe \`a l'infini est irr\'eductible. Nous donnons aussi une r\'eponse compl\`ete \`a la question de l'existence de compactifications de $\cc^2 $ avec des valorisations divisoriels prescrites \`a l'infini. En particulier, nous montrons qu'une \'evaluation sur $\cc(x,y) $ centr\'ee \`a l'infini d\'etermine une compactification de $\cc^2$ ssi il est {\em positivement asym\'etrique} dans le sens de \cite{favsson-tree}.	
\end{abstract}

\section{Introduction} \label{sec-intro}
The topic of this article is compact normal analytic surfaces containing $\cc^2$, henceforth to be called simply {\em compactifications} (of $\cc^2$). Compactifications of $\cc^2$, being one of the most natural and simplest classes of compact surfaces, have been the subject of numerous articles, see e.g.\ \cite{remmert-van-de-ven}, \cite{morrow}, \cite{brentonfication}, \cite{brenton-singular}, \cite{brenton-graph-1}, \cite{miyanishi-zhang}, \cite{furushima}, \cite{ohta}, \cite{kojima}, \cite{koji-hashi}, \cite{favsson-dynamic}. In particular, Kodaira (as part of his classification of surfaces), and independently Morrow \cite{morrow} showed that every nonsingular compactification of $\cc^2$ is {\em rational} (i.e.\ bimeromorphic to $\pp^2$) and can be obtained from $\pp^2$ or some Hirzebruch surface via a sequence of blow-ups and blow-downs. In this article we initiate a program to study these compactifications via studying the {\em curvettes at infinity} - these are germs of curves which are transversal to a {\em curve at infinity} (i.e.\ a curve lying on the complement of $\cc^2$). We analyze parametrizations of images of these curvettes under the bimeromorphic correspondence to $\pp^2$ and use them in two different ways:
\begin{itemize}
\item To study singularities of the compactifications and of the curves at infinity (Sections \ref{primitive-section}, \ref{singular-section}).
\item To study existence of a compactification such that the orders of vanishing along curves at infinity is a prescribed collection of discrete valuations on $\cc(x,y)$ (Section \ref{inter-section}).    
\end{itemize}
In Part II \cite{sub2-2} of this article we use the tools developed here to completely classify compactifications of $\cc^2$ with one (irreducible) curve at infinity. In a subsequent work we plan to emulate this technique to study more general Moishezon surfaces (i.e.\ analytic surfaces which are bimeromorphic to algebraic surfaces).\\

Our first main result is a description of singularities of compactifications of $\cc^2$ and configuration of the curves at infinity. We call a compactification {\em minimal} if none of the irreducible components of the curve at infinity can be (analytically) contracted\footnote{Note that a minimal compactification of $\cc^2$ may {\em not} be a minimal surface, see Example \ref{minimally-non-singular}}.

\begin{thm} \label{singular-theorem}
Let $\bar X$ be a normal analytic compactification of $\cc^2$. Assume that $\bar X \setminus \cc^2$ has $k$ irreducible components $C_1, \ldots, C_k$. Let $\sing(\bar X)$ be the set of singular points of $\bar X$.
\begin{enumerate}[(1)]
\item \label{non-minimal-bound} $|\sing(\bar X)| \leq 2k$.
\item \label{non-minimal-types} $\bar X$ has at most one singular point which is not sandwiched\footnote{An analytic surface $Y$ has a {\em sandwiched} singularity at a point $P$ if there are proper bimeromorphic maps $U'' \to U \to U'$ where $U$ is a neighborhood of $P$ in $Y$ and $U',U''$ are (open subsets of) non-singular surfaces \cite[Remark 1.12]{spinash}. Sandwiched singularities are {\em rational} \cite[Proposition 1.2]{lipman}.}.
\item \label{curve-singularity}
\begin{enumerate}
\item \label{curve-singularity-1} For each $j$, $1 \leq j \leq k$, $C_j$ has an open set isomorphic to $\cc$; in particular, it has at most one singular point. 
\item \label{curve-singularity-2} 
There is at most one $j$ such that $C_j$ has a singular point which is not in $C_i \cap C_j$ for some $i \neq j$. Moreover, if $Q$ is such a point on $C_j$, then $\bar X$ is also singular at $Q$ and $\bigcup_{i\neq j}C_i$ is contractible; in particular, if in addition $k \geq 2$, then $\bar X$ is {\em not} minimal.
\end{enumerate} 
\item \label{minimal-assertion} Assume $\bar X$ is a minimal compactification of $\cc^2$. Then $|\sing(\bar X)| \leq k +1$. Moreover, there is a point $P \in \bar X$ such that
\begin{enumerate}
\item \label{minimal-configuration} $C_i \cap C_j = \{P\}$ for all $i,j$, $1 \leq i < j \leq k$.
\item \label{minimal-C-P} $C_i \setminus \{P\} \cong \cc$ for each $i$.  
\item \label{minimal-singularities} $\left|\sing(\bar X) \setminus \{P\}\right| \leq k$.
\item \label{minimal-quotient} every point in $\sing(\bar X) \setminus \{P\}$ is a cyclic quotient singularity.
\end{enumerate}
\end{enumerate} 
\end{thm}

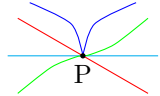
\begin{figure}[htp]
\begin{center}
\begin{tikzpicture}

	\draw [blue] plot [smooth] coordinates {({-cos(45)},{sin(45)}) ({-0.45*cos(70)}, {0.45*sin(70)}) (0, 0) };
	\draw [blue] plot [smooth] coordinates {({cos(45)},{sin(45)}) ({0.45*cos(70)}, {0.45*sin(70)}) (0, 0) };
	\draw [cyan] (-1,0) -- (1,0);
	\draw [green] plot [smooth] coordinates {({-cos(30)},{-sin(30)}) ({-0.45*cos(20)}, {-0.45*sin(20)}) ({-0.1*cos(10)}, {-0.1*sin(10)}) (0, 0) };
	\draw [green] plot [smooth] coordinates {({cos(30)},{sin(30)}) ({0.45*cos(20)}, {0.45*sin(20)}) ({0.1*cos(10)}, {0.1*sin(10)}) (0,0) };
	\draw [red] (-{cos(30)},{sin(30)}) -- ({cos(30)},-{sin(30)});
	
	\fill [] (0,0) circle (1pt);
	 \draw (0,0 )  node [below] {P};

\end{tikzpicture}
\caption{Configuration of curves at infinity on a minimal compactification} \label{fig:minimal-configuration}
\end{center}
\end{figure}

\begin{rem}
\mbox{}
\begin{enumerate}[(a)]
\item Both of the upper bounds for $|\sing(\bar X)|$ of Theorem \ref{singular-theorem} are sharp (see Examples \ref{sharp-example-1} and \ref{sharp-example-2}). Example \ref{minimally-non-singular} shows that the lower bound for $|\sing(\bar X)|$ in both cases is zero.
\item Let $Q$ be a singular point of some $C_j$. Assertion \eqref{curve-singularity-1} implies that $C_j$ has a {\em totally extraordinary singularity} at $Q$ in the language of \cite{brentonfication}. Consequently, assertion \eqref{curve-singularity} improves and generalizes the main result of \cite{brentonfication}. 
\end{enumerate}
\end{rem}

We prove Theorem \ref{singular-theorem} essentially via combinatorial arguments stemming from a careful study of the dual graphs of resolution of singularities of compactifications of $\cc^2$ \footnote{Except for assertion \eqref{curve-singularity-2}, the proof of all assertions of Theorem \ref{singular-theorem} requires only the background material presented in Section \ref{surfacection}. The proof of assertion \eqref{curve-singularity-2} uses Corollary \ref{primitive-corollary} which in turn uses Lemma \ref{normal-lemma}.}. 
The resolution of singularities of a compactification of $\cc^2$ is on the other hand intimately related to the resolution of singularities of generic curvettes at infinity associated to each irreducible curve at infinity. A study of this relation leads us to the second main result (Theorem \ref{priminimal-resolution-graph}) in which we give an explicit description of the dual graph of minimal resolution of singularities of compactifications of $\cc^2$ which are {\em primitive}, i.e.\ for which the curve at infinity is irreducible. As a by product of this description we give new proofs of Jung's theorem on polynomial automorphisms of $\cc^2$ (Corollary \ref{jung}), and Remmert and Van de Ven's result that $\pp^2$ is the only smooth analytic compactification of $\cc^2$ for which the curve at infinity is irreducible (Corollary \ref{rem-de-ven}).\\
%
%
%

A motivation for the work on this article was to understand divisorial valuations centered at infinity on $\cc[x,y]$ - each of these is the order of vanishing along some curve at infinity on some compactification of $\cc^2$. However, these valuations can be explicitly described {\em without} resorting to any compactification, e.g.\ by a finite {\em generating sequence} \cite{spivaluations} of polynomials, or a (finite) sequence of {\em key polynomials} \cite{maclane-key}, or by a {\em Puiseux polynomial} (i.e.\ a Puiseux series with finitely many terms) in $x^{-1}$ or $y^{-1}$ \cite[Chapter 4]{favsson-tree}. The most basic question in this context is:
\begin{bold-question} \label{existential-question}
Assume that we have explicit algebraic description (e.g.\ in one of the equivalent ways mentioned above) of divisorial valuations $\nu_1, \ldots, \nu_k$ on $\cc[x,y]$; in other words, assume that for all polynomials $f \in \cc[x,y]$, we have explicit recipes to compute $\nu_j(f)$, $1 \leq j \leq k$. Determine if there exists a compactification $\bar X$ of $\cc^2$ such that the $\nu_j$'s are precisely the order of vanishing along the curves at infinity on $\bar X$.  
\end{bold-question}
Question \ref{existential-question} is about the existence of a geometric `model' underlying some algebraic data. It follows that the answer should involve interpretation of relevant geometric objects in terms of the input data. Indeed, if $\nu$ is a divisorial valuation on $\cc[x,y]$ associated to a curve $C$ at infinity on some compactification $\bar X$ of $\cc^2$, then the key polynomials of $\nu$ can be used to define `natural' representatives of generic curvettes at infinity associated to $C$ (see Remark \ref{mremark}). Combining this observation with Grauert's characterization of contractible curves (\cite{grauert}, see Theorem \ref{greorem}) we give a complete and explicit answer to Question \ref{existential-question}. Here we give a formulation of this answer in terms of the sequence of key polynomials:\\

Given $\nu_j$'s as in Question \ref{existential-question}, we may (by a generic linear change of coordinates) choose coordinates $(x,y)$ such that $\nu_j(x) < 0$ and $\nu_j(x) \leq \nu_j(y)$ for each $j$. Then set $(u,v) := (1/x,y/x)$, so that each $\nu_j$ is non-negative on $\cc[u,v]$ (with $\nu_j(u) > 0$), and therefore each can be described by a finite sequence of key polynomials. Let $\tilde g_{j,0} = u,\ \tilde g_{j,1} = v, \tilde g_{j,2}, \ldots, \tilde g_{j,l_j} \in \cc[u,v]$ be the sequence of key polynomials of $\nu_j$ (or a {\em minimal generating sequence} in the terminology of \cite{spivaluations}) with respect to $(u,v)$-coordinates. Pick the smallest positive integer $n_{j,l_j}$ such that $n_{j,l_j} \nu_j(\tilde g_{j,l_j})$ is in the semigroup generated by $\nu_j(\tilde g_{j,s})$, $1 \leq s \leq l_j-1$. Then it follows from the property of key polynomials that $n_{j,l_j} \nu_j(\tilde g_{j,l_j}) = \sum_{s=0}^{l_j-1}n_{j,s}\nu_j(\tilde g_{j,s})$ where $n_{j,s}$ are non-negative integers such that $n_{j,s} < \deg_v(\tilde g_{j,s+1})/\deg_v(\tilde g_{j,s})$ for $1 \leq s \leq l_j -1$. Let $\scrM$ be the matrix with entries 
\begin{align*}
m_{ij} 
	&=  d_jn_{j,l_j}\nu_i(u) -\min\left\{ n_{j,l_j}\nu_i(\tilde g_{j,l_j}), \sum_{s=0}^{l_j-1}n_{j,s}\nu_i(\tilde g_{j,s})\right\}   \\
\end{align*}
where $d_j = \deg_v(\tilde g_{j,l_j})$.

\begin{thm} \label{determinanthm}
The answer to question \ref{existential-question} is affirmative iff $\det (-\scrM) < 0$. 
\end{thm}

In the special case that $k=1$, Theorem \ref{determinanthm} implies that a valuation $\nu$ (centered at infinity on $\cc[x,y]$) determines a compactification of $\cc^2$ iff it is {\em positively skewed} in the sense of \cite{favsson-tree}. As the first step to the proof of Theorem \ref{determinanthm} we study a special case of Question \ref{existential-question}, where the answer is affirmative and the resulting compactification {\em dominates} $\pp^2$:

\begin{thm} \label{intersection-thm}
Assume $\nu_1 = -\deg$, where $\deg$ is the degree in $(x,y)$ coordinates. Also assume (w.l.o.g.) that $\nu_i$'s are mutually non-proportional. Then 
\begin{enumerate}
\item \label{easy-existence} There exists a projective (in particular, algebraic) compactification $\bar X$ of $\cc^2$ which affirmatively answers Question \ref{existential-question}. 
\item \label{easy-rational} The singular points of $\bar X$ (if they exist) are sandwiched.
\item \label{easy-intersection} The matrix of intersection numbers of the curves at infinity on $\bar X$ is $\scrM^{-1}$.
\end{enumerate}
\end{thm}

\begin{rem}[Interpretation of the matrix $\scrM$] \label{mremark}
Let $\xi$ be an indeterminate and define 
\begin{align*}
\tilde g_{\nu_j} &:= \tilde g_{j,l_j}^{n_{j,l_j}} - \xi \prod_{s=0}^{l_j-1}\tilde g_{j,s}^{n_{j,s}} \in \cc[u,v,\xi],\\
g_{\nu_j} &:= x^{\deg_v(\tilde g_{\nu_j})} \tilde g_{\nu_j}(1/x,y/x,\xi) \in \cc[x,x^{-1},y,\xi]
\end{align*}
Then it is straightforward to see that $m_{ij} = -\nu_i(g_{\nu_j}(x,y,\tilde \xi))$ for generic $\tilde \xi \in \cc$. Geometrically these $g_{\nu_j}(x,y,\tilde \xi)$'s define {\em generic curvettes at infinity} associated to $\nu_j$ (see Definition \ref{key-definition} and Proposition \ref{curvette-equation}). 
\end{rem}

\begin{rem}
Theorem \ref{intersection-thm} remains valid if $-\nu_1$ is any {\em weighted degree} corresponding to positive weights for $x$ and $y$, or even more generally, if $\nu_1$ is the divisorial valuation associated to the curve at infinity on any primitive compactification of $\cc^2$ with at worst {\em sandwiched} singularities. This follows from essentially the same arguments as in the proof of Theorem \ref{intersection-thm}. 
\end{rem}

\subsection{Organization}
After presenting some background material in Section \ref{background}, we introduce in Section \ref{curvettection} the notion of {\em generic curvettes at infinity} on $\cc^2$ associated to (irreducible) curves at infinity on compactifications on $\cc^2$. In Section \ref{primitive-section} we describe the dual graph of minimal resolution of singularities of primitive compactifications of $\cc^2$ and as corollaries prove Jung's theorem on polynomial automorphisms of $\cc^2$ (Corollary \ref{jung}), and Remmert and Van de Ven's result that $\pp^2$ is the only smooth primitive compactification of $\cc^2$ (Corollary \ref{rem-de-ven}). Section \ref{singular-section} contains the proof of Theorem \ref{singular-theorem} and Section \ref{inter-section} contains the proof of Theorems \ref{determinanthm} and \ref{intersection-thm}.

\subsection{(Un)convention}

In this article we make the unconventional choice to parametrize analytic curves as the parameter approaches infinity (as opposed to zero). We do this because it is more convenient for studying the behaviour of analytic curves on $\cc^2$ as they approach infinity, and studying how the `order of the growth' of these parametrizations is affected by change of coordinates on $\cc^2$. E.g.\ if $f \in \cc[x,y]$ and $L$ is the line $y = ax$, in order to measure the order of growth of $f|_L$ near infinity, we could say 
\begin{itemize}
\item either parametrize $L$ as $t \mapsto (t,at)$ as $t \to \infty$ and compute the degree in $t$ of $f(t,at)$, 
\item or parametrize $L$ as $t \mapsto (t^{-1},at^{-1})$ as $t \to 0$, compute the order in $t$ of $f(t^{-1},at^{-1})$, and take its negative.
\end{itemize}
In this article we chose to adopt the first approach. A consequence of this choice is that instead of using the usual Puiseux series (Definition \ref{meromorphic-defn}) in $t$ where terms appear with increasing order in $t$, we have to use series in $t$ in which terms appear with {\em decreasing} order in $t$; we call these {\em descending Puiseux series} (Definition \ref{dpuiseux}). As a justification of our choice, we invite the reader to formulate Lemma \ref{normal-lemma} (which is a crucial tool in our proof of the results of Section \ref{primitive-section}) using parametrization from a neighborhood of zero and usual Puiseux series, and to compare the resulting formulation with ours.

\subsection{Acknowledgements}

I heartily thank Professor Pierre Milman. This work was done while I was his post-doc at University of Toronto. It was essentially an attempt to understand some of his questions in a simple case and the exposition profited enormously from speaking in his weekly seminar and from his questions. Very special thanks also go to Dmitry Kerner - his questions {\em forced} me to think and formulate the results in geometric and much more understandable terms. Some of the results of this article were announced in \cite{announcement-2}.

\section{Background} \label{background}
\subsection{Puiseux series}

\begin{defn}[Meromorphic Puiseux series] \label{meromorphic-defn}
A {\em meromorphic Puiseux series} in a variable $u$ is a fractional power series of the form $\sum_{m \geq M} a_m u^{m/p}$ for some $m,M \in \zz$, $p  \geq 1$ and $a_m \in \cc$ for all $m \in \zz$. If all exponents of $u$ appearing in a meromorphic Puiseux series are positive, then it is simply called a {\em Puiseux series} (in $u$). Given a meromorphic Puiseux series $\phi(u)$ in $u$, write it in the following form:
\begin{align*} 
\phi(u) =  \cdots+ a_1 u^{\frac{q_1}{p_1}} + \cdots +  a_2 u^{\frac{q_2}{p_1p_2}} + \cdots 
	+ a_l u^{\frac{q_l}{p_1p_2 \cdots p_l}} + \cdots 
\end{align*}
where $q_1/p_1$ is the smallest non-integer exponent, and for each $k$, $1 \leq k \leq l$, we have that $a_k \neq 0$, $p_k \geq 2$, $\gcd(p_k, q_k) = 1$, and the exponents of all terms with order between $\frac{q_k}{p_1\cdots p_k}$ and $\frac{q_k}{p_1\cdots p_{k+1}}$ (or, if $k = l$, then all terms of order $> \frac{1}{p_1\cdots p_l}$) belong to $\frac{1}{p_1 \cdots p_{k}}\zz$. Then the pairs $(q_1, p_1), \ldots, (q_l, p_l)$, are called the {\em Puiseux pairs} of $\phi$ and the exponents $\frac{q_k}{p_1 \cdots p_k}$, $1 \leq k \leq l$, are called {\em characteristic exponents} of $\phi$. The {\em polydromy order} \cite[Chapter 1]{casas-alvero} of $\phi$ is $p := p_1\cdots p_l$, i.e.\ the polydromy order of $\phi$ is the smallest $p$ such that $\phi \in \cc((u^{1/p}))$. Let $\zeta$ be a primitive $p$-th root of unity. Then the {\em conjugates} of $\phi$ are
\begin{align*}
\phi_j(u) :=  \cdots+ a_1 \zeta^{jq_1p_2\cdots p_l} u^{\frac{q_1}{p_1}} + \cdots +  a_2 \zeta^{jq_2p_3\cdots p_l} u^{\frac{q_2}{p_1p_2}} + \cdots 
	+ a_l \zeta^{jq_l} u^{\frac{q_l}{p_1p_2 \cdots p_l}} + \cdots
\end{align*}
for $1 \leq j \leq p$ (i.e.\ $\phi_j$ is constructed by multiplying the coefficients of terms of $\phi$ with order $n/p$ by $\zeta^{jn}$).
\end{defn}

We use the standard fact that the field of meromorphic Puiseux series in $u$ is the algebraic closure of $\cc((u))$:

\begin{thm} \label{puiseux-thm}
Let $f \in \cc((u))[v]$ be an irreducible monic polynomial in $v$ of degree $d$. Then there exists a meromorphic Puiseux series $\phi(u)$ in $u$ of polydromy order $d$ such that 
\begin{align*}
f = \prod_{i=1}^d (v - \phi_i(u)),
\end{align*}
where $\phi_i$'s are conjugates of $\phi$. 
\end{thm}

\begin{defn}[descending Puiseux Series] \label{dpuiseux}
A {\em descending Puiseux series} in $x$ is a meromorphic Puiseux series in $x^{-1}$. The notions regarding meromorphic Puiseux series defined in Definition \ref{meromorphic-defn} extend naturally to the setting of descending Puiseux series. In particular, if $\phi(x)$ is a descending Puiseux series and the Puiseux pairs of $\phi(1/x)$ are $(q_1, p_1), \ldots, (q_l, p_l)$, then $\phi$ has Puiseux pairs $(-q_1, p_1), \ldots, (-q_l, p_l)$, polydromy order $p := p_1 \cdots p_l$, and characteristic exponents $-q_k/(p_1 \cdots p_k)$ for $1 \leq k \leq l$.   
\end{defn}

We use descending Puiseux series via the following result, which is an immediate corollary of Theorem \ref{puiseux-thm}. 

\begin{cor} \label{descending-parametrization}
Let $(x,y)$ be a system of (polynomial) coordinates on $X = \cc^2$. Embed $X \into \pp^2$ via the map $(x,y) \mapsto [1:x:y]$. Let $P = [0:a:b]$ be a point at infinity and $\gamma$ be the germ of an analytic curve at $P$. Assume $a \neq 0$ and $\gamma$ is not the germ of the line at infinity. Then in $(x,y)$-coordinates $\gamma$ has a parametrization of the form $t \mapsto (t, \phi(t))$, $|t| \gg 0$, where $\phi(t)$ is a descending Puiseux series in $t$.  
\end{cor}

\subsection{Divisorial discrete valuations} \label{divisorial-section}
Let $\sigma:Y' \dashrightarrow Y$ be a bimeromorphic correspondence of normal complex algebraic surfaces and $C$ be an irreducible analytic curve on $Y'$. Then the local ring $\sheaf_{Y',C}$ of $C$ on $Y'$ is a discrete valuation ring. Let $\nu$ be the associated valuation on the field $\scrK$ of meromorphic functions on $Y'$; in other words $\nu$ is the order of vanishing along $C$. We say that $\nu$ is a {\em divisorial discrete valuation} on $\scrK$; the {\em center} of $\nu$ on $Y$ is $\sigma(C\setminus S)$, where $S$ is the set of points of indeterminacy of $\sigma$ (the normality of $Y$ ensures that $S$ is a discrete set, so that $C \setminus S \neq \emptyset$). Moreover, if $U$ is an open subset of $Y$, we say that $\nu$ is {\em centered at infinity} with respect to $U$ iff $\sigma(C\setminus S) \subseteq Y\setminus U$. The following result, which connects Puiseux series and divisorial discrete valuations, is a reformulation of \cite[Proposition 4.1]{favsson-tree}.

\begin{thm}\label{puiseux-correspondence}
Let $P \in \sigma(C\setminus S)$. Assume $Y$ is non-singular at $P$. Let $(u,v)$ be an analytic system of coordinates on a neighborhood $U$ of $P$ such that $\nu(u) > 0$. Then there is a {\em Puiseux polynomial} (i.e.\ a Puiseux series with finitely many terms) $\phi_\nu(u)$ (unique up to conjugacy) in $u$ and a (unique) rational number $r_\nu >  \deg_u(\phi_\nu)$ such that for every $f \in \cc[[u,v]]$, 
\begin{align}
\nu(f(u,v)) = \nu(u)\ord_u(f(u, \phi_\nu(u) + \xi u^{r_\nu})), \label{psi-nu-defn}
\end{align}
where $\xi$ is an indeterminate. 
\end{thm}

\begin{rem}[Geometric interpretation of $\phi_\nu(u) + \xi u^{r_\nu}$]
If $Q$ is a generic point of $C \cap \sigma^{-1}(U)$ such that both $Y'$ and $C$ are non-singular at $Q$, and $D$ is an irreducible analytic curve on $Y'$ which intersects $C$ transversally at $Q$, then near $\sigma(Q)$ the (possibly singular) curve $\sigma(D)$ has a Puiseux parametrization of the form $v =  \phi_\nu(u) + \xi' u^{r_\nu} + \hot$, where $\xi' \in \cc$ is generic, and $\hot$ denotes `higher order terms' (in $u$). See Proposition \ref{key-prop}, assertion \ref{key-parametrization} for a more precise statement.
\end{rem}

Combining Theorem \ref{puiseux-correspondence} with Corollary \ref{descending-parametrization} yields:

\begin{cor}\label{descending-correspondence}
Retain the notations and assumptions of Theorem \ref{puiseux-correspondence}. Assume moreover that there exists an open subset $U$ of $Y$ such that 
\begin{enumerate}
\item $\nu$ is centered at infinity with respect to $U$. 
\item there are analytic coordinates $(x,y)$ on $U$ such that $(u,v) = (1/x,y/x)$. 
\end{enumerate} 
Then there is a {\em descending Puiseux polynomial} (i.e.\ a descending Puiseux series with finitely many terms) $\phi_\nu(x)$ (unique up to conjugacy) in $x$ and a (unique) rational number $r_\nu <  \ord_x(\phi_\nu)$ such that for every $f \in \cc[x,y]$, 
\begin{align}
\nu(f(x,y)) = \nu(x)\deg_x(f(x, \phi_\nu(x) + \xi x^{r_\nu})), \label{psi-nu-descending-defn}
\end{align}
where $\xi$ is an indeterminate. 
\end{cor}

\begin{defn} \label{generic-descending-defn}
In the situation of Corollary \ref{descending-correspondence}, we say that $\psi_\nu(x,\xi) := \phi_\nu(x) + \xi x^{r_\nu}$ is the {\em generic descending Puiseux series of $\nu$}. Moreover, if $Y'$ is a surface bimeromorphic to $Y$ and $C \subseteq Y'$ is a curve such that $\nu$ is the order of vanishing along $C$, then we also say that $\psi_\nu(x,\xi)$ is the {\em generic descending Puiseux series associated to $C$}.
\end{defn}

\begin{defn}[Formal Puiseux pairs of generic descending Puiseux series] \label{formal-definition}
Let $\nu$ and $\psi_\nu(x,\xi) = \phi_\nu(x) + \xi x^{r_\nu}$ be as in Definition \ref{generic-descending-defn}. Let the Puiseux pairs of $\phi_\nu$ be $(q_1, p_1), \ldots, (q_l, p_l)$. Express $r_\nu$ as $q_{l+1}/(p_1 \cdots p_lp_{l+1})$, where $p_{l+1} \geq 1$ and $\gcd(q_{l+1}, p_{l+1}) = 1$. Then the {\em formal Puiseux pairs} of $\psi_\nu$ are $(q_1, p_1), \ldots, (q_{l+1}, p_{l+1})$, with $(q_{l+1}, p_{l+1})$ being the {\em generic} formal Puiseux pair. The {\em formal polydromy order} of $\psi_\nu$ is $p := p_1\cdots p_{l+1}$. 
\end{defn}

\subsection{Key polynomials (and generating sequences)}
In addition to Puiseux series, divisorial discrete valuations centered at a non-singular point on a surface can also be described in terms of a (finite) {\em generating sequence} (in the terminology of \cite{spivaluations}) or a (finite) sequence of {\em key polynomials} (in the terminology of \cite{maclane-key}). In this article we use key polynomials; regarding generating sequences, we only point out that every sequence of key polynomials contains a generating sequence \cite[Remark 2.31]{favsson-tree}.\\

Consider the setting of Theorem \ref{puiseux-correspondence}. The key polynomials of $\nu$ with respect to $(u,v)$-coordinates is a finite sequence of polynomials $\tilde g_0 = u,\tilde g_1 = v, \tilde g_2, \ldots, \tilde g_l \in \cc[u,v]$. We refer to \cite[Section 2.1]{favsson-tree} or \cite{maclane-key} for their defining properties. The following proposition is the compilation of all properties of key polynomials that we use.

\begin{prop} \label{key-prop}
Let $U$ be an open neighborhood of $P$ such that $(u,v)$ defines a system of coordinates on $U$. 
\begin{enumerate}
\item \label{key-initial} For each $j \geq 1$, $\tilde g_j$ is of the form 
\begin{align*}
\tilde g_j(u,v) = (v-a)^{d_j} + u\tilde h_j(u,v) 
\end{align*}
where $a \in \cc$ and $\tilde h_j \in \cc[u,v]$ with $\deg_v(\tilde h_j) < d_j$ (where $\deg_v$ denotes the degree in $v$). In particular, $\tilde g_j$ is monic in $v$ of degree $d_j$. Moreover, $d_{j+1}/d_j$ is an integer for each $j$, $1 \leq j \leq l-1$. 
\item For each $j \geq 1$, $\tilde g_j$ is irreducible as an element in $\cc[[u]][v]$. 
\item Let $n_l$ be the smallest positive integer such that $n_l\nu(\tilde g_l)$ is in the semigroup generated by $\nu(\tilde g_0), \ldots, \nu(\tilde g_{l-1})$. Then
\begin{enumerate}
\item There exist (unique) non-negative integers $n_0, \ldots, n_{l-1}$ such that $n_j < d_{j+1}/d_j$ for $1 \leq j \leq l-1$ and $n_l\nu(\tilde g_l) = \sum_{j=0}^{l-1} n_j\nu(\tilde g_j)$. 
\item \label{key-equation} Let $\xi$ be an indeterminate. Define $\tilde g_\nu(u,v,\xi) := \tilde g_l^{n_l} - \xi \prod_{j=0}^{l-1} \tilde g_j^{n_j} \in \cc[u,v,\xi]$. Then there exists a non-empty open disc $\tilde \Delta \subseteq \cc$ such that for all $\tilde \xi \in \tilde \Delta$, the strict transform of the curve $\{\tilde g_\nu(u,v,\tilde \xi) = 0\} \subseteq U$ on $\sigma^{-1}(U)$ intersects $C$ transversally at a single point. 
\item \label{key-parametrization} Let $\phi_\nu(u) + \xi u^{r_\nu}$ be as in \eqref{psi-nu-defn}. Then for all $\tilde \xi \in \tilde \Delta$, $\tilde g_\nu(u,v,\tilde \xi)$ is irreducible in $\cc[[u]][v]$ and has a root $v = \tilde \phi(u)$ where $\tilde \phi(u)$ is a Puiseux series in $u$ of the form
\begin{align*}
\tilde \phi(u) = \phi_\nu(u) + \tilde \xi^{1/n_l} u^{r_\nu} + \hot
\end{align*}
\end{enumerate} 
\end{enumerate}
\end{prop}

\begin{example}
Assume $\sigma:Y' \to Y$ is the minimal resolution of the singularity of the germ of $v^3 - u^2 = 0$ at the origin, and $C \subset Y'$ is the {\em last} exceptional curve. Then key polynomials are $u,v$. Moreover, $\nu(u) = 3$ and $\nu(v) = 2$. Proposition \ref{key-prop} in this case simply says that for generic $\tilde \xi \in \cc$, the strict transform of the germ of $v^3 - \tilde \xi u^2 = 0$ at the origin is transversal to $C$. Similarly, assume $\sigma$ is the minimal resolution of the singularity at the origin of the curve $(v^3 - u^2)^2 - u^3v^2 = 0$, and $C \subset Y'$ is the last exceptional curve. Then key polynomials are $u,v, v^3 - u^2$. Moreover, $\nu(u) = 6$, $\nu(v) = 4$, $\nu(v^3 - u^2) = 13$, and Proposition \ref{key-prop} says that for generic $\tilde \xi \in \cc$, the strict transform of the germ of $(v^3 - u^2)^2 - \tilde \xi u^3v^2 = 0$ at the origin is transversal to $C$. 
\end{example}

\subsection{Theory of surfaces} \label{surfacection}
In this section we compile some facts from bimeromorphic geometry of analytic surfaces. We start with Grauert's criterion for (analytic) contractibility of curves:

\begin{thm}[{\cite{grauert}}] \label{greorem}
Let $Y$ be a smooth complex analytic surface. Let $C_1, \ldots, C_n$ be irreducible curves on $Y$ and $C := C_1 \cup \cdots \cup C_n$. The following are equivalent:
\begin{enumerate}
\item The matrix of intersection numbers $(C_i, C_j)$ is negative definite.
\item There exists a morphism $f: Y \to Z$ such that $Z$ is a normal complex analytic surface, $f(C)$ is a finite set of points and $f|_{Y \setminus C}: Y\setminus C \to Z \setminus f(C)$ is an isomorphism.  
\end{enumerate}
\end{thm}

It is a standard fact that singularities of complex analytic surfaces can be resolved. The singular surfaces $Y'$ we encounter in this article are normal and they come equipped with a bimeromorphic correspondence $\sigma: Y' \dashrightarrow Y$, where $Y$ is a non-singular projective surface. In this case the resolution of singularities of $Y$ is easy to describe:

\begin{thm}
Let $\sigma_0 := \sigma$ and $Y_0 := Y$. Algorithm \ref{alg-resolution} stops after finitely many steps with a bimeromorphic correspondence $\sigma_k: Y' \dashrightarrow Y_k$. Moreover, $\sigma_k^{-1}:Y_k \to Y'$ is a holomorphic map and is a resolution of singularities of $Y'$. 
\end{thm}

\begin{algorithm}[Resolution of singularities of $Y'$] \label{alg-resolution}
Assume $\sigma_i: Y' \dashrightarrow Y_i$ has been defined for $i \geq 0$. If $\sigma_i$ does not contract any curve of $Y'$, then stop. Otherwise pick an irreducible curve $C'$ on $Y'$ which gets contracted to a point $P \in Y_i$. Let $Y_{i+1}$ be the blow up of $Y_i$ at $P$ and $\sigma_{i+1}:Y' \dashrightarrow Y_{i+1}$ be the induced bimeromorphic correspondence. Now repeat. 
\end{algorithm}

We also use the well known fact that every compactification of $\cc^2$ is an {\em algebraic space}, i.e.\ an analytic surface for which  the field of meromorphic functions has transcendence degree $2$:

\begin{thm}[{\cite{morrow}}] \label{moishezon-theorem}
Let $\bar X$ be a normal analytic compactification of $\cc^2$. Then $\bar X$ is an {\em algebraic space}. In particular, the identity map between $\cc^2$ and one of the affine coordinate charts of $\pp^2$ extends to a bimeromorphic correspondence of analytic varieties. 
\end{thm}

\subsection{Dual graph of the resolution of curve singularities} \label{dual-section}

\begin{defn} \label{dual-defn}
Let $E_1, \ldots, E_k$ be non-singular curves on a (non-singular) surface such that for each $i \neq j$, either $E_i \cap E_j = \emptyset$, or $E_i$ and $E_j$ intersect transversally at a single point. Then $E = E_1 \cup \cdots \cup E_k$ is called a {\em simple normal crossing curve}. The {\em (weighted) dual graph} of $E$ is a weighted graph with $k$ vertices $V_1, \ldots, V_k$ such that 
\begin{itemize}
\item there is an edge between $V_i$ and $V_j$ iff $E_i \cap E_j \neq \emptyset$,
\item the weight of $V_i$ is the self intersection number of $E_i$.
\end{itemize}
Usually we will abuse the notation, and label $V_i$'s also by $E_i$. 
\end{defn}

We recall the description of the dual graph of the exceptional divisor of the resolution of an irreducible plane curve singularity following \cite[Section 8.4]{bries-horst-curves}. Assume that we are given an analytically irreducible curve singularity (at a non-singular point of a surface) with Puiseux pairs $(\tilde q_1, \tilde p_1), \ldots, (\tilde q_m, \tilde p_m)$. Then the dual weighted graph for the minimal resolution of the singularity is as in figure \ref{curve-resolution},
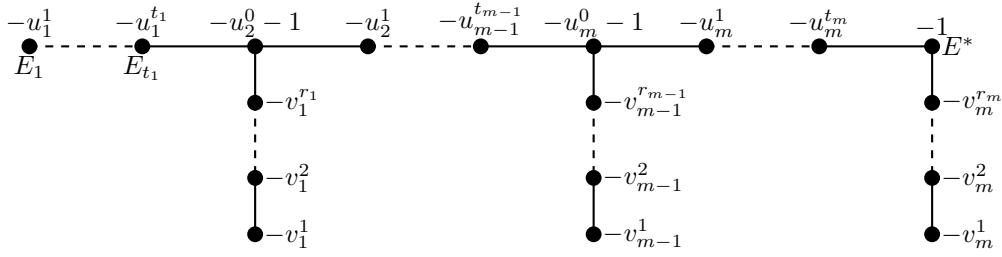
\begin{figure}[htp]
\newcommand{\curveresolutionblock}[6]{
 	\pgfmathsetmacro\x{0}
 	\pgfmathsetmacro\y{0}
 	
 	\draw[thick] (\x - \edge,\y) -- (\x,\y);
 	\draw[thick] (\x,\y) -- (\x+\edge,\y);
 	\draw[thick] (\x,\y) -- (\x, \y-\vedge);
 	\draw[thick, dashed] (\x, \y-\vedge) -- (\x, \y-\vedge - \dashedvedge);
 	\draw[thick] (\x, \y-\vedge - \dashedvedge) -- (\x, \y-2*\vedge - \dashedvedge);
 	
 	\fill[black] (\x - \edge, \y) circle (3pt);
 	\fill[black] (\x, \y) circle (3pt);
 	\fill[black] (\x + \edge, \y) circle (3pt);
 	\fill[black] (\x, \y- \vedge) circle (3pt);
 	\fill[black] (\x, \y- \vedge - \dashedvedge) circle (3pt);
 	\fill[black] (\x, \y- 2*\vedge - \dashedvedge) circle (3pt);
 	
 	\draw (\x- \edge,\y )  node (left) [above] {$-{#1}_{#5}^{{#2}_{#5}}$};
 	\draw (\x,\y )  node (center) [above] {$-{#1}_{#6}^0 - 1$};
 	\draw (\x+ \edge,\y )  node (right) [above] {$-{#1}_{#6}^1$};
 	\draw (\x,\y- \vedge)  node (bottom1) [right] {$-{#3}_{#5}^{{#4}_{#5}}$};
 	\draw (\x,\y- \vedge - \dashedvedge)  node (bottom2) [right] {$-{#3}_{#5}^2$};
 	\draw (\x,\y- 2*\vedge - \dashedvedge)  node (bottom3) [right] {$-{#3}_{#5}^1$};
}
\begin{center}
\begin{tikzpicture}
 	\pgfmathsetmacro\edge{1.5}
 	\pgfmathsetmacro\vedge{.75}
 	\pgfmathsetmacro\dashedvedge{1}
 	
 	\draw[thick, dashed] (0,0) -- (\edge,0);
 	\fill[black] (0, 0) circle (3pt);
 	\draw (0,0)  node (first) [above] {$-u_1^1$};
 	\draw (0,0)  node (e1) [below] {$E_1$};
 	
	\begin{scope}[shift={(2*\edge,0)}]
		\curveresolutionblock{u}{t}{v}{r}{1}{2}
	\end{scope}
	\draw (\edge,0)  node (et1) [below] {$E_{t_1}$};
	
	\draw[thick, dashed] (3*\edge,0) -- (4*\edge,0);
	
	\begin{scope}[shift={(5*\edge,0)}]
		\curveresolutionblock{u}{t}{v}{r}{m-1}{m}
	\end{scope}	
	
	\draw[thick, dashed] (6*\edge,0) -- (7*\edge,0);	
	
	\pgfmathsetmacro\x{8*\edge}
 	\pgfmathsetmacro\y{0}
 	
 	\draw[thick] (\x - \edge,\y) -- (\x,\y);
 	\draw[thick] (\x,\y) -- (\x, \y-\vedge);
 	\draw[thick, dashed] (\x, \y-\vedge) -- (\x, \y-\vedge - \dashedvedge);
 	\draw[thick] (\x, \y-\vedge - \dashedvedge) -- (\x, \y-2*\vedge - \dashedvedge);
 	
 	\fill[black] (\x - \edge, \y) circle (3pt);
 	\fill[black] (\x, \y) circle (3pt);
 	\fill[black] (\x, \y- \vedge) circle (3pt);
 	\fill[black] (\x, \y- \vedge - \dashedvedge) circle (3pt);
 	\fill[black] (\x, \y- 2*\vedge - \dashedvedge) circle (3pt);
 	
 	\draw (\x- \edge,\y )  node (left) [above] {$-u_m^{t_m}$};
 	\draw (\x,\y )  node (center) [above] {$-1$};
	\draw (\x,\y)  node (minus1) [right] {$E^*$}; 
 	\draw (\x,\y- \vedge)  node (bottom1) [right] {$-v_m^{r_m}$};
 	\draw (\x,\y- \vedge - \dashedvedge)  node (bottom2) [right] {$-v_m^{2}$};
 	\draw (\x,\y- 2*\vedge - \dashedvedge)  node (bottom3) [right] {$-v_m^{1}$};
		
\end{tikzpicture}
\caption{Dual graph for the minimal resolution of singularities of an irreducible plane curve-germ}\label{curve-resolution}
\end{center}
\end{figure}
where we denoted the `last exceptional divisor' by $E^*$ and the `left-most' $t_1$ vertices by $E_1, \ldots, E_{t_1}$ (and left all other vertices untitled). The weights $u_i^j$ and $v_i^j$ satisfy: $u_i^0, v_i^0 \geq 1$ and $u_i^j, v_i^j \geq 2$ for $j > 0$, and are uniquely determined from the continued fractions (see, e.g.\ \cite[Section 2.2]{mendris-nemethi}): 
\begin{gather} \label{continued-fractions}
\frac{\tilde p_i}{q'_i} = u_i^0 - \cfrac{1}{u_i^1 - \cfrac{1}{\ddots - \frac{1}{u_i^{t_i}}}},\quad 
\frac{q'_i}{\tilde p_i} = v_i^0 - \cfrac{1}{v_i^1 - \cfrac{1}{\ddots - \frac{1}{v_i^{r_i}}}},\quad 
\text{where}\ q'_i := 
	\begin{cases}
	\tilde q_1 								& \text{if}\ i = 1\\
	\tilde q_i - \tilde q_{i-1}\tilde p_i	& \text{otherwise.}
	\end{cases}
\end{gather}
Note that $(q'_1, \tilde p_1), \ldots, (q'_l, \tilde p_l)$ are called the {\em Newton pairs} of the curve branch, and the Puiseux series of the branch can be expressed in the following form:
$$\psi(u) =  \cdots+ u^{\frac{q'_1}{\tilde p_1}} (a'_1 + \cdots +  u^{\frac{q'_2}{\tilde p_1\tilde p_2}}(a'_2 + \cdots + u^{\frac{q'_3}{\tilde p_1\tilde p_2\tilde p_3}}(\cdots ))).$$

\section{Generic curvettes at infinity} \label{curvettection}
\begin{notation} \label{notation-x}
Throughout the rest of the article we use $X$ to denote $\cc^2$ with coordinate ring $\cc[x,y]$ and $\Xxy$ to denote copy of $\pp^2$ such that $X$ is embedded into $\Xxy$ via the map $(x,y) \mapsto [1:x:y]$. We also denote by $L_\infty$ the {\em line at infinity} $\Xxy\setminus X$, and by $Q_y$ the point of intersection of $L_\infty$ and (closure of) the $y$-axis. 
\end{notation}

\begin{defn}
An {\em irreducible analytic curve germ at infinity} on $X$ is the image $\gamma$ of an analytic map $\eta$ from a punctured neighborhood $\Delta'$ of the origin in $\cc$ to $X$ such that $|\eta(s)| \to \infty$ as $|s| \to 0$ (in other words, $\eta$ is analytic on $\Delta'$ and has a pole at the origin). Let $\bar X$ be an analytic compactification of $X$. Theorem \ref{moishezon-theorem} implies that there is a unique point $P \in \bar X \setminus X$ such that $|\eta(s)| \to P$ as $|s| \to 0$. We call $P$ the {\em center} of $\gamma$ on $\bar X$, and write $P = \lim_{\bar X} \gamma$. Let $\Xxy$ be as in Notation \ref{notation-x}. Assume $\lim_{\Xxy} \gamma \neq Q_y$. Then Corollary \ref{descending-parametrization} implies that for $|t| \gg 0$, $\gamma$ has a parametrization of the form $\theta: t \mapsto (t, \phi(t))$, where $\phi(t)$ is a descending Puiseux series in $t$. We call $\theta$ a {\em descending Puiseux parametrization} of $\gamma$. 
\end{defn}

\begin{example}
Note that if $\lim_{\Xxy} \gamma = Q_y$, then $\gamma$ might {\em not} have descending Puiseux parametrization. Indeed, let $\gamma$ be the curve-germ at infinity on $X$ corresponding to the germ of the (closure of the) $y$-axis at $Q_y$. Then there is no descending Puiseux series $\phi(t)$ in $t$ such that $\gamma$ has a parametrization of the form $t \mapsto (t, \phi(t))$ for $|t| \gg 0$.
\end{example}

Now let $\bar X$ be a normal analytic compactification of $X$ and $C$ be an irreducible component of the curve at $\bar X_\infty := \bar X\setminus X$ at infinity on $\bar X$. Theorem \ref{moishezon-theorem} implies that the identity map of $X$ induces a bimeromorphic correspondence $\sigmaxy: \bar X \dashrightarrow \Xxy$. Let $S$ be the set of points of indeterminacy of $\sigmaxy$. Since $\bar X$ is normal, it follows that $S$ is a finite set. After a linear change of coordinates of $\cc[x,y]$, we may ensure that $\bar X$ satisfies the property \sumption for every irreducible curve $C \subseteq \bar X\setminus X$:
\begin{align}
\parbox{.57\textwidth}{
$\sigmaxy(C \setminus S) \neq \{Q_y\}$ (i.e.\ either $\sigmaxy$ does not contract $C$, or it contracts $C$ to some point other than $Q_y$).
}\tag{$C_{(x,y)}$}
\label{assumption}
\end{align}

\begin{remtation} \label{assumption-remark}
Note that if $C$ is an irreducible curve in $\bar X \setminus X$ and $\nu$ is the order of vanishing along $C$, then 
\begin{align*}
\bar X\ \text{satisfies \sumption}
	&\iff \sigmaxy(C \setminus S) \neq \{Q_y\} \\
	&\iff y/x\ \text{restricts to a regular function on a non-empty open set of $C$} \\
	&\iff \nu(y/x) \geq 0 \\
	&\iff \nu(x) \leq \nu(y).
\end{align*}
\end{remtation}

Pick $P \in \sigmaxy(C\setminus S)\setminus\{Q_y\} \subseteq L_\infty$. 
Let $\gamma$ be an irreducible curve-germ at infinity on $X$ with $\lim_{\Xxy}\gamma = P$. Let $P_\gamma := \lim_{\bar X}  \gamma \in \bar X$ and $\bar \gamma^{\bar X} := \gamma \cup \{P_\gamma\}$ be the closure of $\gamma$ in $\bar X$. We say that $\gamma$ is a {\em curvette at infinity} \footnote{The use of the term `curvette' to denote germs of transversal curves at smooth points of a given curve is due to Deligne \cite{sga7II}.} associated to $C$ iff $P_\gamma \in C$ and $\bar \gamma^{\bar X}$ intersects $C$ transversally at $P_\gamma$ (in particular, $P_\gamma$ is a non-singular point of both $C$ and $\bar \gamma^{\bar X}$). We say that $\gamma$ is a {\em generic curvette at infinity} associated to $C$ if furthermore $P_\gamma$ is a generic point of $C$.

\begin{prop}[Parametrizations of generic curvettes at infinity]\label{curvette-prop}
Let $\gamma$ be a generic curvette at infinity associated to $C$ and let $t \mapsto (t, \phi(t))$ be a descending Puiseux parametrization of $\gamma$.
\begin{enumerate}
\item There is a unique rational number $r$ and a finite set $\scrE \subseteq C$ 
such that if $\tilde \gamma$ is a curvette at infinity on $X$, then $\lim_{\bar X} \tilde \gamma \in C \setminus \scrE$ iff $\tilde \gamma$ has a descending Puiseux parametrization of the form $t \mapsto (t,\tilde \phi(t))$ such that $\deg_t(\tilde \phi(t) - \phi(t)) = r$.
\item Let $\tilde \gamma$ be a curvette at infinity on $X$ with a descending Puiseux parametrization of the form $t \mapsto (t,\tilde \phi(t))$ such that $\deg_t(\tilde \phi(t) - \phi(t)) \leq r$. Write $\phi - \tilde \phi = \tilde \xi x^{r} + \lot$ where $\tilde \xi \in \cc$. Then
\begin{enumerate}
\item \label{limiting-curves-2-1} $\lim_{\bar X} \tilde \gamma$ depends only on $\tilde \xi$. In particular, for generic values of $\tilde \xi$, $\lim_{\bar X} \tilde \gamma$ is a generic element of $C$.
\item \label{non-singular-limit-1} $\lim_{\bar X} \tilde \gamma$ is a non-singular point of $\bar {\tilde\gamma}^{\bar X}$ iff there are no characteristic exponents of $\tilde \phi$ smaller than $r$. 
\item \label{non-singular-limit-2} for all but finitely many values of $\tilde \xi$, $\tilde \gamma$ is a curvette at infinity associated to $C$ iff either (and therefore, both!) of the properties of assertion \ref{non-singular-limit-1} is satisfied. 
\end{enumerate}
\item Let $[\phi]_{>r}(x)$ be the descending Puiseux polynomial in $x$ obtained by removing from $\phi(x)$ all terms with degree $\leq r$ and define $\psi(x,\xi) := [\phi]_{>r}(x) + \xi x^r$, where $\xi$ is an indeterminate. Then $\psi(x,\xi)$ is precisely the generic descending Puiseux series of $\nu$.
\end{enumerate}
\end{prop}

\begin{proof}
The relation between (generic) descending Puiseux series and key polynomials of a valuation is given by assertion \ref{key-parametrization} of Proposition \ref{key-prop}. Proposition \ref{curvette-prop} follows from interpreting the properties of key polynomials compiled in Proposition \ref{key-prop} in terms of the associated descending Puiseux series. 
\end{proof}

Set $(u,v) := (1/x,y/x)$ and let $U$ be the coordinate chart of $\Xxy$ with coordinates $(u,v)$. Consider the situation of Corollary \ref{descending-correspondence} with $\sigma = \sigmaxy$. 

\begin{defn} \label{key-definition}
Let $\tilde g_0,\ldots, \tilde g_l \in \cc[u,v]$ be the sequence of key polynomials of $\nu$ with respect to $(u,v)$-coordinates. Set 
\begin{align*}
g_i &:= \begin{cases}
			x 	& \text{if}\ i = 0,\\
			x^{\deg_v(\tilde g_i)}\tilde g_i(1/x,y/x) & \text{othewise.}
		\end{cases}
\end{align*}
For each $i \geq 1$, $g_i \in \cc[x,x^{-1},y]$ and it is monic in $y$. We call $g_i$'s the sequence of {\em key forms} of $\nu$ with respect to $(x,y)$-coordinates. Finally, let $n_0, \ldots, n_l$ be as in Proposition \ref{key-prop}. Then define 
\begin{align*}
g_\nu(x,y,\xi) &:= x^{\deg_v(\tilde g_\nu)} \tilde g_\nu(1/x,y/x,\xi) 
		= g_l^{n_l} - \xi x^{n'_0} \prod_{j=1}^{l-1} g_j^{n_j} 
		\in \cc[x,x^{-1},y,\xi]
\end{align*}
where $n'_0 = n_l\deg_v(g_l) - n_0 - \sum_{j=1}^{l-1}n_j\deg_v(g_j)$. We call $g_\nu(x,y,\xi)$ the {\em generic key form} of $\nu$ in $(x,y)$-coordinates. 
\end{defn}

\begin{prop}[Affine equations of generic curvettes at infinity] \label{curvette-equation}
Pick $n \geq 0$ such that $x^n g_\nu \in \cc[x,y,\xi]$. For all $\tilde \xi \in \cc$, let $Z_{\tilde \xi}$ be the closure in $\Xxy$ of the curve $\{x^n g_\nu(x,y,\tilde \xi) = 0\} \subseteq X$.
\begin{enumerate}
\item \label{curvette-pt} For each $\tilde \xi \in \cc$, $Z_{\tilde \xi}$ intersects $L_\infty\setminus \{Q_y\}$ at a single point $Q_{\tilde \xi}$.
\item \label{curvette-assertion} For generic $\tilde \xi \in \cc$, the germ of $Z_{\tilde \xi} $ in a punctured neighborhood of $Q_{\tilde \xi}$ is a curvette at infinity associated to $C$.
\item \label{extra-pt} $Z_{\tilde \xi}$ intersects $L_\infty$ at $Q_y$ with intersection multiplicity $n$. 
\end{enumerate}
\end{prop}

\begin{proof}
Assertions \eqref{curvette-pt} and \eqref{extra-pt} follow from assertion \ref{key-initial} of Proposition \ref{key-prop}, and assertion \ref{curvette-assertion} follows from assertion \eqref{key-equation} of Proposition \ref{key-prop}.
\end{proof}

\begin{example} \label{degrexample}
Let $\bar X = \Xxy$ and $C = L_\infty$ (so that $\nu$ is the negative of degree in $(x,y)$-coordinates). Then the key forms are $x,y$, and the generic key form is $y - \xi x$. Propostion \ref{curvette-equation} in this case simply states (the obvious fact) that for $y - \tilde \xi x$ intersects the line $L_\infty$ transversally for generic $\tilde \xi$.  
\end{example}

\subsection{Effect of automorphisms of $\cc^2$ on parametrizations of generic curvettes at infinity} \label{normal-section}
Let $\gamma$ be a curve-germ at infinity on $X$ with a descending Puiseux parametrization $t \mapsto (t,\phi(t))$. In this section we study the effect on $\deg_t(\phi(t))$ of two `simple' types of automorphisms of the plane described below; the (simple) observations made in this section will be crucial in our proof of Jung's theorem that these automorphisms generate the full group of polynomial automorphisms of $\cc^2$.

\begin{defn}
let $F: \cc[x,y] \to \cc[x,y]$ be an automorphism. We call $F$ a {\em Type I automorphism} if it is of the form $(x,y) \mapsto (y,x)$ and a {\em Type II automorphism} if it is of the form $(x,y) \mapsto (x, y + ax^n)$, where $a \in \cc$ and $n \geq 0$.
\end{defn}

\begin{lemma} \label{typed-change-of-coordinates}
Let $\gamma$ be a curve-germ at infinity on $X$ with a descending Puiseux parametrization $t \mapsto (t,\phi(t))$ and $\omega := \deg_t(\phi(t))$, i.e.\ 
$$\phi(t) = at^\omega + \lot$$
for some $a \in \cc$. Assume $\omega > 0$.
\begin{enumerate}
\item  
\begin{enumerate}
\item \label{typeIchange1}  After the type I automorphism $(x,y) \mapsto (y,x)$, $\gamma$ has a descending Puiseux parametrization $t \mapsto (t,\tilde \phi(t))$ where $\deg_t(\tilde \phi(t)) = 1/\omega$. 
\item \label{typeIchange2}  Moreover, if $\omega = 1/n$ for some integer $n \geq 2$, then the number of Puiseux pairs of $\tilde \phi(t)$ is one less than the number of Puiseux pairs of $\phi(t)$. 
\end{enumerate}
\item \label{typeIIchange} If $\omega$ is a non-negative integer, then after the type II automorphism $(x,y) \mapsto (x,y-ax^\omega)$, $\gamma$ has a descending Puiseux parametrization of the form $t \mapsto (t,\phi(t) - at^\omega)$. 
\end{enumerate}
\end{lemma}

\begin{proof}
Assertions \eqref{typeIchange1} and \eqref{typeIIchange} are easy to see. Assertion \eqref{typeIchange2} follows from a straightforward induction on the number of Puiseux pairs of $\phi$. 
\end{proof}

Let $\bar X$ be a compactification of $X$ and $C$ be an irreducible component of the curve at infinity on $\bar X$. The following lemma shows that after a composition of finitely many Type I and II automorphisms, we can ensure that generic curvettes associated to $C$ have descending Puiseux parametrizations, and the initial term of these parametrizations has a `normal form'. 

\begin{lemma} \label{normal-lemma}
Let $\bar X$ and $C$ be as above, and $\gamma$ be a generic curvette at infinity on $X$ associated to $C$.
After a finite sequence of Type I and Type II automorphisms of $\cc[x,y]$, we can ensure that $\gamma$ has a descending Puiseux parametrization $t \mapsto (t,\phi(t))$, where $\phi(t)$ is of the following form:
\begin{align} \label{normal-form}
\phi(t) = \begin{cases} 
			\tilde \xi t^r, & r \in \qq,~ r \leq 1,~ \tilde \xi \in \cc\ \text{is generic, or} \\
			a_1t^{\omega_1} + \lot, & a_1 \in \cc\setminus\{0\},\ \omega_1 \in \qq \setminus \left(\zz_{\geq 0} \cup \{1/n: n \in \nn\}\right),~ \omega_1 < 1.
		   \end{cases}
\end{align}
\end{lemma}

\begin{proof}
Since any linear change of coordinates of $\cc[x,y]$ is a composition of Type I and II automorphisms, it follows that after composition of finitely many Type I and II automorphisms, we can ensure that $\bar X$ satisfies \sumption, which implies in particular that $\gamma$ has a descending Puiseux parametrization $t \mapsto (t,\phi(t))$. Assertion \eqref{typeIchange1} of Lemma \ref{typed-change-of-coordinates} then implies that it suffices to prove the following statement: after a  a finite sequence of automorphisms of $\cc[x,y]$ of types I and II, we can ensure that $\phi(t)$ is {\em not} of the following form:
\begin{align} \label{bad-form}
a_1t^{\omega_1} + \lot,\ \text{where}\ a_1 \in \cc\setminus\{0\},\ \text{and}\ \omega_1 \in \zz_{\geq 0} \cup \{1/n: n \in \zz_{\geq 1}\}. \tag{!}
\end{align}
Indeed, assume $\phi(t)$ is of the form \eqref{bad-form}. Then either $\phi(t) = at^n + \lot$ for some polynomial $f(x) \in \cc[x]$, or $\phi(t) = at^{1/n} + \lot$ for some $a \neq 0$ and a positive integer $n > 1$. In the first case apply Type II automorphism $(x,y) \mapsto (x, y - ax^n)$ and in the second case apply the Type I automorphism $(x,y) \mapsto (y,x)$. Note that
\begin{compactenum}
\item in the second case the number of Puiseux pairs of $\phi(t)$ decreases by one (assertion \eqref{typeIchange2} of Lemma \ref{typed-change-of-coordinates}),
\item in the first case the number of Puiseux pairs of $\phi(t)$ does not change, but $\deg_t(\phi(t))$ decreases (assertion \eqref{typeIIchange} of Lemma \ref{typed-change-of-coordinates}).
\end{compactenum}
The above observations imply that this process ends after finitely many steps, as required to complete the proof of the lemma. 
\end{proof}

\begin{reminition} \label{normal-form-definition}
We say that the initial exponent of $\phi(t)$ is in the {\em normal form} if $\phi(t)$ is as in \eqref{normal-form}. Note that $\phi(t)$ is in the normal form iff either $\sigmaxy$ maps $C$ generically on to $L_\infty \subseteq \Xxy$ (in which case $r=1$), or contracts $C$ to the point of intersection of $L_\infty$ and $x$-axis. 
\end{reminition}

\begin{rem} 
With a bit of more work than the proof of Lemma \ref{normal-lemma}, it can be shown that there is a `normal form' for $\phi(t)$ itself (i.e.\ not only the initial exponent). In \cite{sub2-2} we use this normal form to compute the moduli spaces and groups of automorphisms of algebraic compactifications of $\cc^2$ with one irreducible curve at infinity.
\end{rem}

\section{Primitive compactifications and resolution of their singularities} \label{primitive-section}
\begin{defn} \label{augmented-defn}
Let $\pi: \tilde X \to \bar X$ be a resolution of singularities of a compactification $\bar X$ of $X \cong \cc^2$ such that $\tilde X \setminus X$ is a simple normal crossing curve. The {\em augmented dual graph} of $\pi$ is the dual graph (Definition \ref{dual-defn}) of $\tilde X \setminus X$. 
\end{defn}

Let $\bar X$ be a (normal analytic) compactification of $\cc^2$ which is {\em primitive}, i.e.\ the curve $C$ at infinity on $\bar X$ is irreducible. In this section we show that the minimal resolution of singularities of $\bar X$ satisfies the properties of Definition \ref{augmented-defn}, and describe its augmented dual graph. As a consequence, we derive a new proof of Remmert and Van de Ven's characterization of $\pp^2$ as the only non-singular primitive compactification of $\cc^2$ and Jung's theorem on polynomial automorphisms.\\

We continue to adopt Notation \ref{notation-x} and assume that $\bar X$ satisfies \sumption, i.e.\ there exists $P \in \sigmaxy(C \setminus S) \setminus\{Q_y\}$. Let the generic descending Puiseux series for $C$ be 
\begin{align*} 
\psi (x,\xi) 
	&=  \phi(x) + \xi x^r\\
	&=  \cdots+ a_1 x^{\frac{q_1}{p_1}} + \cdots +  a_2 x^{\frac{q_2}{p_1p_2}} + \cdots 
	+ a_l x^{\frac{q_l}{p_1p_2 \cdots p_l}} + \cdots + \xi x^{\frac{q_l}{p_1p_2 \cdots p_{l+1}}}
\end{align*}
where $(q_1, p_1), \ldots, (q_{l+1}, p_{l+1})$ are the {\em formal Puiseux pairs} (Definition \ref{formal-definition}) of $\psi$. Then $(u,v) = (1/x,y/x)$ is a system of coordinate near $P$, and Proposition \ref{curvette-prop} implies that generic curvettes at infinity associated to $C$ have Puiseux parametrizations of the form
\begin{align}
v 	&= \cdots+ a_1 u^{\frac{\tilde q_1}{\tilde p_1}} + \cdots +  a_2 u^{\frac{\tilde q_2}{\tilde p_1 \tilde p_2}} + \cdots 
	+ a_l u^{\frac{\tilde q_l}{\tilde p_1 \tilde p_2 \cdots \tilde p_l}} 
	+ \cdots + \tilde \xi u^{\frac{\tilde q_l}{\tilde p_1 \tilde p_2 \cdots \tilde p_{l+1}}} + \text{\hot}
	\label{puiseux-param}
\end{align}
where $(\tilde q_i, \tilde p_i) = (p_1 \cdots p_i - q_i, p_i)$, $1 \leq i \leq l+1$, and $\tilde \xi$ is a generic element of $\cc$. Apply Algorithm \ref{alg-resolution} with $\sigma_0 = \sigmaxy$ to construct a resolution of singularities $\tilde \sigma: \tilde X \to \bar X$. Let $\Gamma$ be the corresponding augmented dual graph. The following proposition gives a description of $\Gamma$ in terms of the dual graph of the minimal resolution of the plane curve singularity of the curve germ with Puiseux parametrization \eqref{puiseux-param}. 

\begin{prop} \label{primitive-resolution-graph}
Let $E$ be the strict transform of $C$. Assume the initial exponent of $\psi$ is in the normal form (Definition \ref{normal-form-definition}). Then
\begin{enumerate}
\item If $\psi(x,\xi) = \xi x$, then $\bar X \overset{\sigmaxy}{\cong} \Xxy \cong \pp^2$ (in particular, $\bar X$ is non-singular), and $\Gamma$ consists of a single vertex $E$. 
\item \label{case:graph1} Otherwise if $p_{l+1} > 1$, then $\Gamma$ is as in figure \ref{fig:graph1}, where $\Gamma'$ is as in figure \ref{curve-resolution} with $m = l + 1$. In particular, $\bar X$ has at most two singular points, one of them is at worst a cyclic quotient singularity.
\item \label{case:graph2} Otherwise ($p_{l+1} = 1$ and) $\Gamma$ is as in figure \ref{fig:graph2}, where $\Gamma''$ is the graph of Figure \ref{curve-resolution} with $m = l$ and one change - namely the self-intersection number of $E^*$ in $\Gamma''$ is $-2$. In particular, $\bar X$ has at most one singular point. 
\end{enumerate}
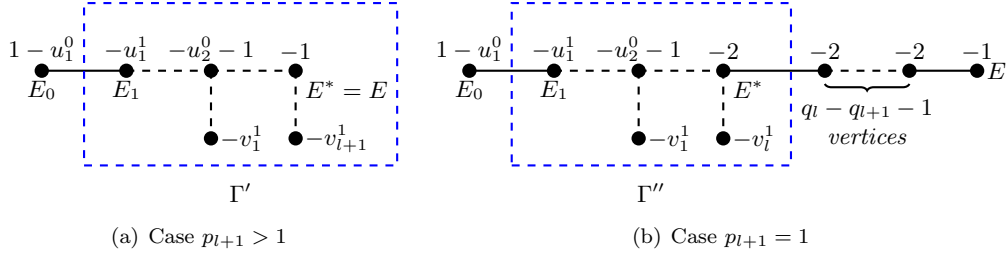
\begin{figure}[htp]
\newcommand{\curveresolutionsmallblock}[5]{
	\pgfmathsetmacro\edge{1.25}
 	\pgfmathsetmacro\vedge{1}
 	\pgfmathsetmacro\x{0}
 	\pgfmathsetmacro\y{0}
 	
 	\draw[thick] (\x - \edge,\y) -- (\x,\y);
 	\draw[thick, dashed] (\x,\y) -- (\x+\edge,\y);
 	\draw[thick, dashed] (\x + \edge,\y) -- (\x+\edge,\y - \vedge);
 	\draw[thick, dashed] (\x + \edge, \y) -- (\x + 2*\edge, \y);
 	\draw[thick, dashed] (\x + 2*\edge,\y) -- (\x+ 2*\edge,\y - \vedge);
 	
 	\fill[black] (\x - \edge, \y) circle (3pt);
 	\fill[black] (\x, \y) circle (3pt);
 	\fill[black] (\x + \edge, \y) circle (3pt);
 	\fill[black] (\x + \edge, \y - \vedge) circle (3pt);
 	\fill[black] (\x + 2*\edge, \y) circle (3pt);
 	\fill[black] (\x + 2*\edge, \y - \vedge) circle (3pt);
 	
 	\draw (\x- \edge,\y )  node (e0up) [above] {$1 - u_1^0$};
 	\draw (\x- \edge,\y )  node (e0down) [below] {$E_0$};
 	\draw (\x,\y )  node (e1up) [above] {$-u_1^1$};
 	\draw (\x,\y )  node (e1down) [below] {$E_1$};
 	\draw (\x+ \edge,\y )  node (middleup) [above] {$-u_2^0-1$};
 	\draw (\x+ \edge,\y - \vedge)  node (middlebottom) [right] {$-v_1^1$}; 	
 	\draw (\x+ 2*\edge,\y )  node (estarup) [above] {#1};
 	\draw (\x+ 2*\edge,\y )  node (estardown) [below right] {#2};
 	\draw (\x+ 2*\edge, \y - \vedge)  node (rightbottom) [right] {$-v_{#3}^1$};
 	
 	\draw[thick, dashed, blue] (\x - 0.5*\edge, \y + \vedge) rectangle (\x + 2*\edge + #5, \y - 1.4*\vedge);
 	
 	\draw (\x + 0.75*\edge + 0.5*#5, \y - 1.8*\vedge) node (Gamma) {#4};
}
\begin{center}
\subfigure[Case $p_{l+1} > 1$]{
	\label{fig:graph1}
	\begin{tikzpicture}[scale=.9, font = \small] 	
 		\curveresolutionsmallblock{$-1$}{$E^* = E$}{l+1}{$\Gamma'$}{1.5}
	\end{tikzpicture}
}
\subfigure[Case $p_{l+1} = 1$]{
	\label{fig:graph2}
	\begin{tikzpicture}[scale=.9, font = \small] 	
 		\curveresolutionsmallblock{$-2$}{$E^*$}{l}{$\Gamma''$}{1}
 		\pgfmathsetmacro\edge{1.25}
 		\pgfmathsetmacro\extralen{0.5}
 		\begin{scope}[shift={(2*\edge + \extralen,0)}]
 			\pgfmathsetmacro\sedge{1}
 			\pgfmathsetmacro\dashedge{1.25}
 			\draw[thick] (-\extralen, 0) -- (\sedge, 0);
 			\draw[thick, dashed] (\sedge, 0) -- (\sedge + \dashedge, 0);
 			\draw[thick] (\sedge + \dashedge, 0) -- (2*\sedge + \dashedge, 0);
 		
 			\fill[black] (\sedge, 0) circle (3pt);
 			\fill[black] (\sedge + \dashedge, 0) circle (3pt);
 			\fill[black] (2*\sedge + \dashedge, 0) circle (3pt);
 		
 			\draw (\sedge,0 )  node (firstup) [above] {$-2$};
 			\draw (\sedge + \dashedge,0 )  node (secondup) [above] {$-2$};
 			\draw (2*\sedge + \dashedge,0 )  node (lastup) [above] {$-1$};
 			\draw (2*\sedge + \dashedge,0 )  node (lastright) [right] {$E$};
 			
 			\draw [thick, decoration={brace, mirror, raise=5pt},decorate] (\sedge,0) -- (\sedge + \dashedge,0);
 			\draw (\sedge + 0.5*\dashedge,-0.75) node [text width= 1.75cm, align = center] (extranodes) {$q_l - q_{l+1} -1$ vertices};
 		\end{scope}	
	\end{tikzpicture}
}
\caption{Augmented dual graph for the resolution of Algorithm \ref{alg-resolution}}\label{fig:primitive-resolution-graph}
\end{center}
\end{figure}
\end{prop} 

\begin{rem} \label{non-minimal-remark}
Note that the resolution of Proposition \ref{primitive-resolution-graph} is {\em not} minimal if (and only if) $u_1^0 = 2$.
\end{rem}

\begin{proof}[Proof of Proposition \ref{primitive-resolution-graph}]
The first assertion is straightforward. The other assertions follow from the discussion in Section \ref{dual-section} and the following observations:\\

1. In the scenario of assertion \ref{case:graph1}, the Puiseux pairs of generic curvettes at infinity associated to $C$ are $(\tilde q_1, \tilde p_1), \ldots, (\tilde q_{l+1}, \tilde p_{l+1})$ and Algorithm \ref{alg-resolution} corresponds precisely to resolution of singularities of these curvettes at infinity.\\

2. In the scenario of assertion \ref{case:graph2}, the Puiseux pairs of generic curvettes at infinity associated to $C$ are $(\tilde q_1, \tilde p_1), \ldots, (\tilde q_{l}, \tilde p_{l})$ and Algorithm \ref{alg-resolution} corresponds to at first resolving the singularities of these curvettes at infinity, and then $q_{l} - q_{l+1}$ additional blow-ups. \\

3. The vertex $e_0$ in Figures \ref{fig:graph1} and \ref{fig:graph2} corresponds to $E_0$, which is the strict transform of $L_\infty \subseteq \Xxy$. The equation of $L_\infty$ near $P$ is $u=0$. On the other hand, the normal form of $\psi$ implies that the order (in $u$) of the right hand side of \eqref{puiseux-param} is $\tilde q_1/\tilde p_1$. It follows that  strict transform of $L_\infty$ contains the center of precisely the first $u^0_1$-blow ups (where $u^0_1$ is defined in \eqref{continued-fractions}).
\end{proof}

\begin{rem} \label{easy-2-remark}
More generally, if $\bar X$ is an arbitrary normal analytic compactification of $\cc^2$ and $C$ is an irreducible curve at infinity on $\bar X$, then the arguments from the proof of Proposition \ref{primitive-resolution-graph} imply that there is a non-singular compactification $\tilde X$ of $\cc^2$ dominating $\Xxy \cong \pp^2$ such that the dual graph of the curve at infinity on $\tilde X$ has the same shape as $\Gamma$ of figure \ref{fig:primitive-resolution-graph}. In particular, contracting all curves at infinity on $\bar X$ other than $E_0$ and $E$ results in a compactification $\bar X^*$ with precisely two irreducible curves $E^*_0$ and $E^*$ at infinity, 
\begin{itemize}
\item the bimeromorphic correspondence $\Xxy \dashrightarrow \bar X^*$ maps $L_\infty$ dominantly on to $E^*_0$.
\item the bimeromorphic correspondence $\bar X \dashrightarrow \bar X^*$ maps $C$ dominantly on to $E^*$.
\end{itemize}
This implies that 
\begin{enumerate}
\item $\bar X^*$ is precisely the compactification guaranteed by assertion \ref{easy-existence} of Theorem \ref{intersection-thm} in the case that $k = 2$ and $\nu_2$ is the divisorial valuation associated to $C$.
\item $\tilde X$ is precisely the {\em minimal} resolution of singularities of $\bar X^*$.
\end{enumerate}
Moreover, let $P^*$ be the point of intersection of $E^*_0$ and $E^*$. We claim that $E^* \setminus \{P^*\} \cong \cc$. Indeed, this is clear if $\Gamma$ is as in figure \ref{fig:graph2}. On the other hand, if $\Gamma$ is as in figure \ref{fig:graph1}, then it suffices to show that $E^*$ is non-singular at the point $Q^*$ to which the curves corresponding to the right-most vertical string of $\Gamma$ contracts. But the singularity at $Q^*$ is a cyclic quotient (or {\em Hirzebruch-Jung}) singularity, and $E$ is transversal to the string of exceptional divisor of its resolution. It then follows from the well known properties of cyclic-quotient singularities (see e.g.\ \cite[Section III.5]{barth-peters-van-de-ven}) that $E^*$ does not acquire any singularity at $Q^*$.
\end{rem}

As mentioned in Remark \ref{non-minimal-remark}, the resolution of singularities of $\bar X$ constructed in Proposition \ref{priminimal-resolution-graph} may {\em not} be minimal. Understanding the minimal resolution of $\bar X$ requires a more detailed analysis of the change of the initial exponent of a Puiseux series under blow up. This is the content of the next theorem. 

\begin{thm} \label{priminimal-resolution-graph}
Let the assumptions and notations be as in Proposition \ref{primitive-resolution-graph}; in particular the initial exponent of $\psi$ is in the normal form, and $\Gamma$ is as in figure \ref{fig:graph1} if $p_{l+1} > 1$ and as in figure \ref{fig:graph2} if $p_{l+1} = 1$. 
\begin{enumerate}
\item \label{primitive-nonsingular} $\bar X$ is non-singular iff $\psi(x,\xi) = \xi x$. 
\item \label{primitive-singular-1} Otherwise if $q_1/p_1 > 1/2$, then $u_1^0 > 2$ and $\Gamma$ is the augmented dual graph of the minimal resolution of $\bar X$.
\item \label{primitive-singular-2} Otherwise let $\tilde q_1 := p_1 - q_1$. Then we must have $p_1 = \tilde q_1 + r_1$ and $\tilde q_1 = m_1r_1 + r_2$ for some positive integers $r_1, m_1, r_2$ with $r_2 < r_1 < \tilde q_1$. Moreover, if $t_1$ is as in $\Gamma$ (see the `leftmost' string of figure \ref{curve-resolution}), then $t_1 \geq m_1$ and $u_1^{m_1} \geq 3$ and $u_1^j = 2$ for all $j$, $1 \leq j < m_1$. The augmented dual graph of the minimal resolution of $\bar X$ is gotten from $\Gamma$ by deleting all the vertices to the left of $e_{m_1}$ and changing the weight of $e_{m_1}$ to $-u_1^{m_1} + 1$.
\end{enumerate}
\end{thm}

\begin{proof}
The $(\Leftarrow)$ implication of the first assertion follows from Proposition \ref{primitive-resolution-graph}. Now we assume that $\psi(x,\xi) \neq \xi x$ and show that either 2nd or the 3rd assertion of the theorem is true. Note that this will also prove the $(\im)$ implication of assertion \eqref{primitive-nonsingular} (since a surface is non-singular iff the dual graph of the minimal resolution of singularity is non-empty) and complete the proof of the theorem. \\

Since the initial exponent of $\psi$ is in the normal form, it follows that $\deg_x(\psi) < 1$. We now divide our proof based on different possibilities for $\deg_x(\psi)$. For each case we construct the minimal resolution $\tilde X^{min}$ of singularities of $\bar X$ and show that the exceptional divisor of the morphism $\tilde X^{min} \to \bar X$ is of the required form.\\

\paragraph{\bf Case 1: $\deg_x(\psi) = 1/n$, $n \geq 2$.} In this case $\psi = \xi x^{1/n}$. Consequently \eqref{puiseux-param} implies that a generic curvette $\gamma$ associated to $C$ has Puiseux expansion near $P$ of the form
\begin{align*} 
u = \xi' v^{n/(n-1)} + \hot 
\end{align*}
for a generic $\xi' \in \cc$. Let $\bar X_0 = \Xxy, \bar X_1, \ldots $ be the sequence of surfaces constructed in the resolution Algorithm \ref{alg-resolution}. Then it follows that the strict transform of $\gamma$ on $\bar X_i$ has a Puiseux expansion of the form 
\begin{align*} 
v_i = \xi' u_i^{n-i} + \hot 
\end{align*}
where $(u_i,v_i) := (u/v, v^i/u^{i-1})$. In particular, the bimeromorphic correspondence $\bar X \dashrightarrow \bar X_i$ maps $C$ to the point $(u_i,v_i) = 0$ for $i < n$, and dominantly on to the line $u_n = 0$ (which is precisely the exceptional divisor of the last blow up) for $i = n$. It follows that $\tilde X = \bar X_n$ is precisely the resolution of singularity of $\bar X$ achieved via algorithm \ref{alg-resolution} with the augmented dual graph as in figure \ref{fig:graph-1/n}. 
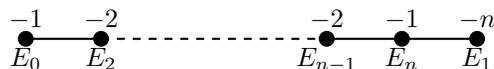
\begin{figure}[htp]

\begin{center}
\begin{tikzpicture}
 	\pgfmathsetmacro\sedge{1}
 	\pgfmathsetmacro\dashedge{3}
 	
 	\fill[black] (0, 0) circle (3pt);
 	\fill[black] (\sedge, 0) circle (3pt);
 	\fill[black] (\sedge + \dashedge, 0) circle (3pt);
 	\fill[black] (2*\sedge + \dashedge, 0) circle (3pt);
 	\fill[black] (3*\sedge + \dashedge, 0) circle (3pt);
 	
 	\draw[thick] (0,0) -- (\sedge,0);
 	\draw[thick, dashed] (\sedge,0) -- (\sedge + \dashedge,0);
 	\draw[thick] (\sedge+\dashedge,0) -- (2*\sedge + \dashedge,0);
 	\draw[thick] (2*\sedge+\dashedge,0) -- (3*\sedge + \dashedge,0);
 	
	\draw (0,0)  node (e0-up) [above] {$-1$};
	\draw (0,0)  node (e0-down) [below] {$E_0$};
	\draw (\sedge,0)  node (e2-up) [above] {$-2$};
	\draw (\sedge,0)  node (e2-down) [below] {$E_2$};
	\draw (\sedge + \dashedge,0)  node (e-n-1-up) [above] {$-2$};
	\draw (\sedge + \dashedge,0)  node (e-n-1-down) [below] {$E_{n-1}$};
	\draw (2*\sedge + \dashedge,0)  node (e-n-up) [above] {$-1$};
	\draw (2*\sedge + \dashedge,0)  node (e-n-down) [below] {$E_n$};
	\draw (3*\sedge + \dashedge,0)  node (e1-up) [above] {$-n$};
	\draw (3*\sedge + \dashedge,0)  node (e1-down) [below] {$E_1$};
\end{tikzpicture}
\caption{Augmented dual graph for resolution when $\deg_x(\psi_\nu) = 1/n$, $n \geq 2$}\label{fig:graph-1/n}
\end{center}
\end{figure}

Since $E_n$ is precisely the pre-image of $C$, it follows that the exceptional divisor of the resolution $\tilde \sigma: \tilde X \to \bar X$ is $\tilde E := E_0 \cup \cdots \cup E_{n-1}$. Note that $\tilde E$ has two connected components: $E_1$ and $\tilde E_1 := E_0 \cup E_2 \cup \cdots \cup E_{n-1}$. By Castelnuovo's criterion $\tilde X^{min}$ is formed from $\tilde X$ by contracting $\tilde E_1$ to a non-singular point. In particular, the exceptional divisor of the minimal resolution $\tilde X^{min} \to \bar X$ is precisely (the isomorphic image of) $E_1$. It is straightforward to check that this is precisely the form of the exceptional divisor prescribed by assertion \ref{primitive-singular-2} of the theorem.\\

\paragraph{\bf Case 2: $1 > \deg_x(\psi) > 1/2$.} In this case a generic curvette $\gamma$ associated to $C$ has Puiseux expansion near $P$ of the form
\begin{align*} 
u = a v^{\alpha} + \hot
\end{align*}
where $a \neq 0 \in \cc$ and $\alpha > 2$. It follows that Algorithm \ref{alg-resolution} requires at least $3$ blow ups, and strict transforms of $L_\infty$ contain the centers of at least the first three blow ups. In particular, the intersection number of the strict transform of $L_\infty$ on $\tilde X$ (which is precisely negative of the label of the vertex $e_0$ in figure \ref{fig:primitive-resolution-graph}) is $\leq -2$. This implies that all the irreducible curves with support in the exceptional divisor of $\tilde \sigma: \tilde X \to \bar X$ has self intersection $\leq -2$. Consequently, $\tilde \sigma$ is precisely the minimal resolution of singularities of $\bar X$ and assertion \ref{primitive-singular-1} of the theorem holds. \\ 

\paragraph{\bf Case 3: $0 < \deg_x(\psi) < 1/2$, $\deg_x(\psi) \neq 1/n$ for all $n \in \zz$.} The hypothesis of this case implies that a generic curvette $\gamma$ associated to $C$ has Puiseux expansion near $P$ of the form
\begin{align*} 
u = a v^{\alpha} + \hot
\end{align*}
where $a \neq 0 \in \cc$ and $1 < \alpha < 2$ such that $\alpha \neq (n+1)/n$ for all $n \geq 1$. Note that $\alpha = p_1/\tilde q_1$ where $\tilde q_1$ is as in \eqref{puiseux-param}. In particular $p_1, \tilde q_1$ are integers with no common factors. Let us follow the steps of the computation of $\gcd(p_1, \tilde q_1) = 1$ via Euclidean algorithm. The assumptions on $\alpha$ translate to the following observations: 
\begin{align*}
p_1 &= \tilde q_1 + r_1 && \text{for some}\ r_1 \in \zz,\ 1 < r_1 < q_1,\ \text{and}\\ 
\tilde q_1 &= m_1r_1 + r_2 && \text{for some}\ m_1, r_2 \in \zz,\ 1 \leq m_1,\ 1 \leq r_2 < r_1.\\ 
\intertext{Let the next step of the computation of $\gcd(p_1,q_1)$ be}
r_1 &= m_2r_2 + r_3 && \text{for some}\ m_2 > 0,\ \text{and}\ 0 \leq r_3 < r_2.
\end{align*}
Then straightforward arguments as in Case 1 shows that after $m_1 + m_2 +1$ blow-ups the dual graph of the union of strict transforms of $E_i$'s for $1 \leq i \leq m_1 + m_2+1$ on $\bar X_{m_1+m_2+1}$ is as in figure \ref{fig:euclid-3-0}, and the Puiseux expansion for the strict transform $\gamma_{m_1+m_2+1}$ of $\gamma$ on $\bar X_{m_1 +m_2+1}$ is given by:
$$u_{m_1+m_2+1} = a'(v_{m_1+m_2+1})^{r_3/r_2} + \hot \quad \text{for some}\ a' \neq 0 \in \cc.$$
where $(u_{m_1+m_2+1},v_{m_1+m_2+1}) = (u^{1+ m_1m_2}/v^{1+ m_2 + m_1m_2}, v^{m_1+1}/u^{m_1})$. Moreover, $u_{m_1 + m_2 + 1} = 0$ and $v_{m_1 + m_2 + 1} = 0$ are respectively the local equations of the strict transform of $E_{m_1+1}$ and $E_{m_1+m_2 +1}$ near $\gamma_{m_1+m_2+1}$. We divide the rest of the proof for this case into the following two subcases:\\

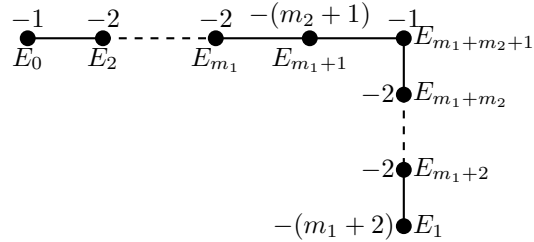
\begin{figure}[htp]

\begin{center}

\begin{tikzpicture}
 	\pgfmathsetmacro\sedge{1}
 	\pgfmathsetmacro\dashedge{1.5}
 	\pgfmathsetmacro\edge{1.25}
 	\pgfmathsetmacro\vedge{.75}
 	\pgfmathsetmacro\dashedvedge{1}
 	
 	\fill[black] (0, 0) circle (3pt);
 	\fill[black] (\sedge,0) circle (3pt);
 	\fill[black] (\sedge + \dashedge, 0) circle (3pt);
 	\fill[black] (\sedge + \edge + \dashedge, 0) circle (3pt);
 	\fill[black] (\sedge + 2*\edge + \dashedge, 0) circle (3pt);
 	\fill[black] (\sedge + 2*\edge + \dashedge, -\vedge) circle (3pt);
 	\fill[black] (\sedge + 2*\edge + \dashedge, -\vedge - \dashedvedge) circle (3pt);
 	\fill[black] (\sedge + 2*\edge + \dashedge, -2*\vedge - \dashedvedge) circle (3pt);
 	
 	\draw[thick] (0,0) -- (\sedge,0);
 	\draw[thick, dashed] (\sedge,0) -- (\sedge + \dashedge,0);
 	\draw[thick] (\sedge + \dashedge,0) -- (\sedge + 2*\edge + \dashedge,0);
 	\draw[thick] (\sedge + 2*\edge + \dashedge,0) -- (\sedge + 2*\edge + \dashedge, -\vedge);
 	\draw[thick, dashed] (\sedge + 2*\edge + \dashedge,-\vedge) -- (\sedge + 2*\edge + \dashedge, -\vedge - \dashedvedge);
 	\draw[thick] (\sedge + 2*\edge + \dashedge,-\vedge - \dashedvedge) -- (\sedge + 2*\edge + \dashedge, -2*\vedge - \dashedvedge);
 	
 	\draw (0,0)  node (E-0-up) [above] {$-1$};
	\draw (0,0)  node (E-0-down) [below] {$E_0$};
	\draw (\sedge,0)  node (E-2-up) [above] {$-2$};
	\draw (\sedge,0)  node (E-2-down) [below] {$E_2$};
	\draw (\sedge + \dashedge,0)  node (E-m1-up) [above] {$-2$};
	\draw (\sedge + \dashedge,0)  node (E-m1-down) [below] {$E_{m_1}$};
	\draw (\sedge + \edge + \dashedge,0)  node (E-m1+1-up) [above] {$-(m_2 + 1)$};
	\draw (\sedge + \edge + \dashedge,0)  node (E-m1+1-down) [below] {$E_{m_1+1}$};
	\draw (\sedge + 2*\edge + \dashedge,0)  node (E-m1+m2+1-up) [above] {$-1$};
	\draw (\sedge + 2*\edge + \dashedge,0)  node (E-m1+m2+1-down) [right] {$E_{m_1+m_2+1}$};
	
	\draw (\sedge + 2*\edge + \dashedge, -\vedge) node (E-m1+m2-up) [left] {$-2$};
	\draw (\sedge + 2*\edge + \dashedge, -\vedge) node (E-m1+m2-down) [right] {$E_{m_1+m_2}$};	
	\draw (\sedge + 2*\edge + \dashedge, -\vedge - \dashedvedge) node (E-m1+2-up) [left] {$-2$};
	\draw (\sedge + 2*\edge + \dashedge, -\vedge - \dashedvedge) node (E-m1+2-down) [right] {$E_{m_1 + 2}$};
	\draw (\sedge + 2*\edge + \dashedge, -2*\vedge - \dashedvedge) node (E-1-up) [left] {$-(m_1+2)$};
	\draw (\sedge + 2*\edge + \dashedge, -2*\vedge - \dashedvedge)  node (E-1-down) [right] {$E_{1}$};
	
\end{tikzpicture}

\caption{Dual graph of $E_0 \cup \cdots \cup E_{m_1+m_2+1}$ after $m_1+m_2+1$ blow-ups} \label{fig:euclid-3-0}
\end{center}
\end{figure}

\paragraph{\bf Subcase 3.1: $r_3 = 0$.} Since $a' \neq 0$, this implies that $\gamma_{m_1 +m_2+1}$ does {\em not} belong to the strict transform of $E_{m_1+1}$ on $\bar X_{m_1+m_2+1}$. It follows from Algorithm \ref{alg-resolution} all the remaining blow-ups for the construction of $\tilde X$ keep (the strict transforms of) $E_0,\ldots, E_{m_1+1}$ unchanged and the dual graph of the exceptional divisor of the morphism $\tilde X \to \bar X$ is of the form as in figure \ref{fig:euclid-3-1-1}. Moreover, $r_3=0$ implies that $r_2 = \gcd(p_1, q_1) = 1$, so that $m_2 = r_1 \geq 2$. The same arguments as in Case 1 then show that the dual graph of the exceptional divisor of the minimal resolution $\tilde X^{min} \to \bar X$ is of the form as in figure \ref{fig:euclid-3-1-2}. This is precisely the form of the dual graph prescribed by assertion \ref{primitive-singular-2} of the theorem.\\

\begin{figure}[htp]

\begin{center}
\subfigure[Non-minimal resolution]{
\label{fig:euclid-3-1-1}
\begin{tikzpicture}	
	[outline/.style={draw=#1,fill=#1!10}]

 	\pgfmathsetmacro\sedge{1}
 	\pgfmathsetmacro\edge{1.25}
 	\pgfmathsetmacro\nedge{2}
 	\pgfmathsetmacro\dashedge{1.5}
 	\pgfmathsetmacro\vedge{.75}
 	\pgfmathsetmacro\dashedvedge{1}
 	\pgfmathsetmacro\vnedge{1}
 	
 	\fill[black] (0, 0) circle (3pt);
 	\fill[black] (\sedge,0) circle (3pt);
 	\fill[black] (\sedge + \dashedge, 0) circle (3pt);
 	\fill[black] (\sedge + \dashedge + \edge, 0) circle (3pt);
 	\fill[black] (\sedge + \dashedge + \edge + \nedge, -\vnedge) circle (3pt);
 	\fill[black] (\sedge + \dashedge + \edge + \nedge, -\vnedge - \dashedvedge) circle (3pt);
 	\fill[black] (\sedge + \dashedge + \edge + \nedge, -\vnedge - \dashedvedge - \vedge) circle (3pt);
 	
 	\draw[thick] (0,0) -- (\sedge,0);
 	\draw[thick, dashed] (\sedge,0) -- (\sedge + \dashedge,0);
 	\draw[thick] (\sedge + \dashedge,0) -- (\sedge + \dashedge + \edge,0);
 	\draw[thick] (\sedge + \dashedge + \edge,0) -- (\sedge + \dashedge + \edge + \nedge, 0);
 	\draw[thick] (\sedge + \dashedge + \edge + \nedge,0) -- (\sedge + \dashedge + \edge + \nedge, -\vnedge);
 	\draw[thick, dashed] (\sedge + \dashedge + \edge + \nedge,-\vnedge) -- (\sedge + \dashedge + \edge + \nedge, -\vnedge - \dashedvedge);
 	\draw[thick] (\sedge + \dashedge + \edge + \nedge,-\vnedge - \dashedvedge) -- (\sedge + \dashedge + \edge + \nedge, -\vnedge - \dashedvedge - \vedge);
 	
 	\draw (0,0)  node (E-0-up) [above] {$-1$};
	\draw (0,0)  node (E-0-down) [below] {$E_0$};
	\draw (\sedge,0)  node (E-2-up) [above] {$-2$};
	\draw (\sedge,0)  node (E-2-down) [below] {$E_2$};
	\draw (\sedge + \dashedge,0)  node (E-m1-up) [above] {$-2$};
	\draw (\sedge + \dashedge,0)  node (E-m1-down) [below] {$E_{m_1}$};
	\draw (\sedge + \edge + \dashedge,0)  node (E-m1+1-up) [above] {$-(m_2 + 1)$};
	\draw (\sedge + \edge + \dashedge,0)  node (E-m1+1-down) [below] {$E_{m_1+1}$};
	\draw (\sedge + \edge + \dashedge + \nedge,0)  node (the-rest) [outline=blue, text width = 1.75cm] {rest of\\ the graph};
	
	\draw (\sedge + \dashedge + \edge + \nedge, -\vnedge) node (E-m1+m2-up) [left] {$-2$};
	\draw (\sedge + \dashedge + \edge + \nedge, -\vnedge) node (E-m1+m2-down) [right] {$E_{m_1+m_2}$};	
	\draw (\sedge + \dashedge + \edge + \nedge, -\vnedge - \dashedvedge) node (E-m1+2-up) [left] {$-2$};
	\draw (\sedge + \dashedge + \edge + \nedge, -\vnedge - \dashedvedge) node (E-m1+2-down) [right] {$E_{m_1 + 2}$};
	\draw (\sedge + \dashedge + \edge + \nedge, -\vnedge - \dashedvedge -\vedge) node (E-1-up) [left] {$-(m_1+2)$};
	\draw (\sedge + \dashedge + \edge + \nedge, -\vnedge - \dashedvedge -\vedge)  node (E-1-down) [right] {$E_{1}$};
\end{tikzpicture}
}
\subfigure[Minimal resolution]{
\label{fig:euclid-3-1-2}
\begin{tikzpicture}	
	[outline/.style={draw=#1,fill=#1!10}]

 	\pgfmathsetmacro\sedge{1}
 	\pgfmathsetmacro\edge{1.25}
 	\pgfmathsetmacro\nedge{2}
 	\pgfmathsetmacro\dashedge{1.5}
 	\pgfmathsetmacro\vedge{.75}
 	\pgfmathsetmacro\dashedvedge{1}
 	\pgfmathsetmacro\vnedge{1}
 	
 	\begin{scope}[shift={(-\sedge - \dashedge - \edge,0)}]
 	
 	\fill[black] (\sedge + \dashedge + \edge, 0) circle (3pt);
 	\fill[black] (\sedge + \dashedge + \edge + \nedge, -\vnedge) circle (3pt);
 	\fill[black] (\sedge + \dashedge + \edge + \nedge, -\vnedge - \dashedvedge) circle (3pt);
 	\fill[black] (\sedge + \dashedge + \edge + \nedge, -\vnedge - \dashedvedge - \vedge) circle (3pt);
 	
 	\draw[thick] (\sedge + \dashedge + \edge,0) -- (\sedge + \dashedge + \edge + \nedge, 0);
 	\draw[thick] (\sedge + \dashedge + \edge + \nedge,0) -- (\sedge + \dashedge + \edge + \nedge, -\vnedge);
 	\draw[thick, dashed] (\sedge + \dashedge + \edge + \nedge,-\vnedge) -- (\sedge + \dashedge + \edge + \nedge, -\vnedge - \dashedvedge);
 	\draw[thick] (\sedge + \dashedge + \edge + \nedge,-\vnedge - \dashedvedge) -- (\sedge + \dashedge + \edge + \nedge, -\vnedge - \dashedvedge - \vedge);

	\draw (\sedge + \edge + \dashedge,0)  node (E-m1+1-up) [above] {$-m_2$};
	\draw (\sedge + \edge + \dashedge,0)  node (E-m1+1-down) [below] {$E_{m_1+1}$};
	\draw (\sedge + \edge + \dashedge + \nedge,0)  node (the-rest) [outline=blue, text width = 1.75cm] {rest of\\ the graph};
	
	\draw (\sedge + \dashedge + \edge + \nedge, -\vnedge) node (E-m1+m2-up) [left] {$-2$};
	\draw (\sedge + \dashedge + \edge + \nedge, -\vnedge) node (E-m1+m2-down) [right] {$E_{m_1+m_2}$};	
	\draw (\sedge + \dashedge + \edge + \nedge, -\vnedge - \dashedvedge) node (E-m1+2-up) [left] {$-2$};
	\draw (\sedge + \dashedge + \edge + \nedge, -\vnedge - \dashedvedge) node (E-m1+2-down) [right] {$E_{m_1 + 2}$};
	\draw (\sedge + \dashedge + \edge + \nedge, -\vnedge - \dashedvedge -\vedge) node (E-1-up) [left] {$-(m_1+2)$};
	\draw (\sedge + \dashedge + \edge + \nedge, -\vnedge - \dashedvedge -\vedge)  node (E-1-down) [right] {$E_{1}$};

	\end{scope}
\end{tikzpicture}
}
\caption{Dual graphs of the exceptional divisor of the resolution of singularities of $\bar X$ for the case $r_3 = 0$} \label{fig:euclid-3-1}
\end{center}
\end{figure}
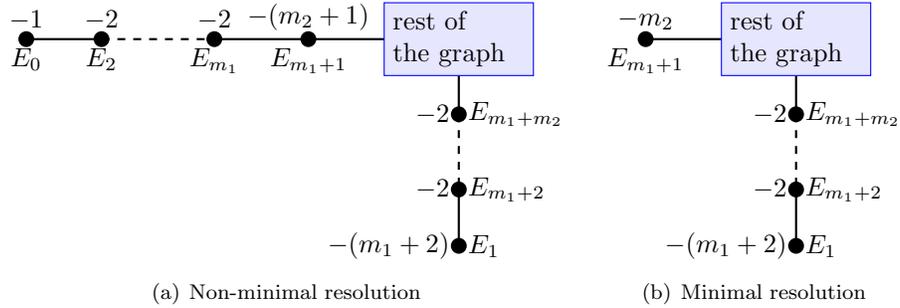

\paragraph{\bf Subcase 3.2. $r_3 > 0$:} In this case $\gamma_{m_1 + m_2 +1}$ intersects the point $P_{m_1+m_2+1}$ of intersection of $E_{m_1 + m_2+1}$ and the strict transform of $E_{m_1+1}$ on $\bar X_{m_1+m_2+1}$. It follows that the bimeromorphic correspondence $\bar X \dashrightarrow \bar X_{m_1+m_2+1}$ maps $C$ to $P_{m_1+m_2+1}$ and therefore Algorithm \ref{alg-resolution} requires at least one more blow up to construct $\tilde X$. The dual graph of the union of strict transforms of $E_i$'s for $1 \leq i \leq m_1 + m_2+2$ on $\bar X_{m_1+m_2+2}$ is as in figure \ref{fig:euclid-3-2-1}. Also, since $r_3 < r_2$, it follows that the strict transform of $\gamma$ on $\bar X_{m_1 +m_2+2}$ does {\em not} intersect the strict transform of $E_{m_1+1}$, and the same reasoning as in Subcase 3.1 then implies that the dual graph of the exceptional divisor of the minimal resolution $\tilde X^{min} \to \bar X$ is of the form as in figure \ref{fig:euclid-3-2-2}. It is straightforward to check that this agrees with assertion \ref{primitive-singular-2}, which completes the proof of the theorem. 
\end{proof}

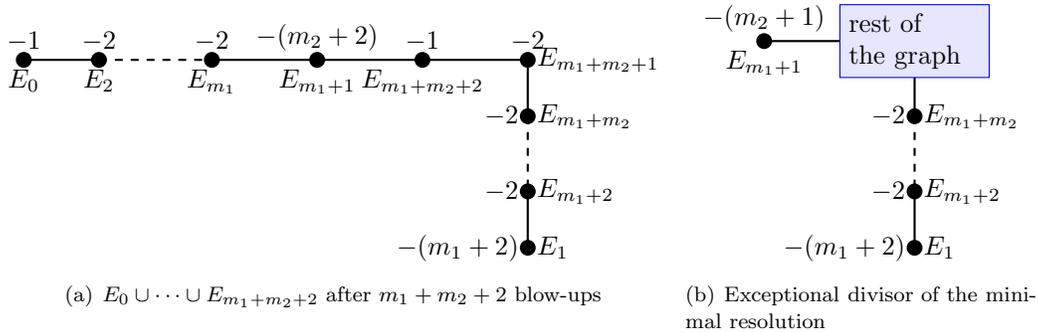
\begin{figure}[htp]
\begin{center}

\subfigure[$E_0 \cup \cdots \cup E_{m_1+m_2+2}$ after $m_1+m_2+2$ blow-ups]{
\label{fig:euclid-3-2-1}
\begin{tikzpicture}
 	\pgfmathsetmacro\sedge{1}
 	\pgfmathsetmacro\dashedge{1.5}
 	\pgfmathsetmacro\edge{1.4}
 	\pgfmathsetmacro\vedge{.75}
 	\pgfmathsetmacro\dashedvedge{1}
 	
 	\fill[black] (0, 0) circle (3pt);
 	\fill[black] (\sedge,0) circle (3pt);
 	\fill[black] (\sedge + \dashedge, 0) circle (3pt);
 	\fill[black] (\sedge + \edge + \dashedge, 0) circle (3pt);
 	\fill[black] (\sedge + 2*\edge + \dashedge, 0) circle (3pt);
 	\fill[black] (\sedge + 3*\edge + \dashedge, 0) circle (3pt);
 	\fill[black] (\sedge + 3*\edge + \dashedge, -\vedge) circle (3pt);
 	\fill[black] (\sedge + 3*\edge + \dashedge, -\vedge - \dashedvedge) circle (3pt);
 	\fill[black] (\sedge + 3*\edge + \dashedge, -2*\vedge - \dashedvedge) circle (3pt);
 	
 	\draw[thick] (0,0) -- (\sedge,0);
 	\draw[thick, dashed] (\sedge,0) -- (\sedge + \dashedge,0);
 	\draw[thick] (\sedge + \dashedge,0) -- (\sedge + \edge + \dashedge,0);
 	\draw[thick] (\sedge + \edge + \dashedge,0) -- (\sedge + 2*\edge + \dashedge,0);
 	\draw[thick] (\sedge + 2*\edge + \dashedge,0) -- (\sedge + 3*\edge + \dashedge, 0);
 	\draw[thick] (\sedge + 3*\edge + \dashedge,0) -- (\sedge + 3*\edge + \dashedge, -\vedge);
 	\draw[thick, dashed] (\sedge + 3*\edge + \dashedge,-\vedge) -- (\sedge + 3*\edge + \dashedge, -\vedge - \dashedvedge);
 	\draw[thick] (\sedge + 3*\edge + \dashedge,-\vedge - \dashedvedge) -- (\sedge + 3*\edge + \dashedge, -2*\vedge - \dashedvedge);
 	
 	\draw (0,0)  node (E-0-up) [above] {$-1$};
	\draw (0,0)  node (E-0-down) [below] {$E_0$};
	\draw (\sedge,0)  node (E-2-up) [above] {$-2$};
	\draw (\sedge,0)  node (E-2-down) [below] {$E_2$};
	\draw (\sedge + \dashedge,0)  node (E-m1-up) [above] {$-2$};
	\draw (\sedge + \dashedge,0)  node (E-m1-down) [below] {$E_{m_1}$};
	\draw (\sedge + \edge + \dashedge,0)  node (E-m1+1-up) [above] {$-(m_2 + 2)$};
	\draw (\sedge + \edge + \dashedge,0)  node (E-m1+1-down) [below] {$E_{m_1+1}$};
	\draw (\sedge + 2*\edge + \dashedge,0)  node (E-m1+m2+2-up) [above] {$-1$};
	\draw (\sedge + 2*\edge + \dashedge,0)  node (E-m1+m2+2-down) [below] {$E_{m_1+m_2+2}$};
	\draw (\sedge + 3*\edge + \dashedge,0)  node (E-m1+m2+1-up) [above] {$-2$};
	\draw (\sedge + 3*\edge + \dashedge,0)  node (E-m1+m2+1-down) [right] {$E_{m_1+m_2+1}$};
	
	\draw (\sedge + 3*\edge + \dashedge, -\vedge) node (E-m1+m2-up) [left] {$-2$};
	\draw (\sedge + 3*\edge + \dashedge, -\vedge) node (E-m1+m2-down) [right] {$E_{m_1+m_2}$};	
	\draw (\sedge + 3*\edge + \dashedge, -\vedge - \dashedvedge) node (E-m1+2-up) [left] {$-2$};
	\draw (\sedge + 3*\edge + \dashedge, -\vedge - \dashedvedge) node (E-m1+2-down) [right] {$E_{m_1 + 2}$};
	\draw (\sedge + 3*\edge + \dashedge, -2*\vedge - \dashedvedge) node (E-1-up) [left] {$-(m_1+2)$};
	\draw (\sedge + 3*\edge + \dashedge, -2*\vedge - \dashedvedge)  node (E-1-down) [right] {$E_{1}$};
	
\end{tikzpicture}
}
\subfigure[Exceptional divisor of the minimal resolution]{
\label{fig:euclid-3-2-2}
\begin{tikzpicture}	
	[outline/.style={draw=#1,fill=#1!10}]

 	\pgfmathsetmacro\sedge{1}
 	\pgfmathsetmacro\edge{1.25}
 	\pgfmathsetmacro\nedge{2}
 	\pgfmathsetmacro\dashedge{1.5}
 	\pgfmathsetmacro\vedge{.75}
 	\pgfmathsetmacro\dashedvedge{1}
 	\pgfmathsetmacro\vnedge{1}
 	
 	\begin{scope}[shift={(-\sedge - \dashedge - \edge,0)}]
 	
 	\fill[black] (\sedge + \dashedge + \edge, 0) circle (3pt);
 	\fill[black] (\sedge + \dashedge + \edge + \nedge, -\vnedge) circle (3pt);
 	\fill[black] (\sedge + \dashedge + \edge + \nedge, -\vnedge - \dashedvedge) circle (3pt);
 	\fill[black] (\sedge + \dashedge + \edge + \nedge, -\vnedge - \dashedvedge - \vedge) circle (3pt);
 	
 	\draw[thick] (\sedge + \dashedge + \edge,0) -- (\sedge + \dashedge + \edge + \nedge, 0);
 	\draw[thick] (\sedge + \dashedge + \edge + \nedge,0) -- (\sedge + \dashedge + \edge + \nedge, -\vnedge);
 	\draw[thick, dashed] (\sedge + \dashedge + \edge + \nedge,-\vnedge) -- (\sedge + \dashedge + \edge + \nedge, -\vnedge - \dashedvedge);
 	\draw[thick] (\sedge + \dashedge + \edge + \nedge,-\vnedge - \dashedvedge) -- (\sedge + \dashedge + \edge + \nedge, -\vnedge - \dashedvedge - \vedge);

	\draw (\sedge + \edge + \dashedge,0)  node (E-m1+1-up) [above] {$-(m_2+1)$};
	\draw (\sedge + \edge + \dashedge,0)  node (E-m1+1-down) [below] {$E_{m_1+1}$};
	\draw (\sedge + \edge + \dashedge + \nedge,0)  node (the-rest) [outline=blue, text width = 1.75cm] {rest of\\ the graph};
	
	\draw (\sedge + \dashedge + \edge + \nedge, -\vnedge) node (E-m1+m2-up) [left] {$-2$};
	\draw (\sedge + \dashedge + \edge + \nedge, -\vnedge) node (E-m1+m2-down) [right] {$E_{m_1+m_2}$};	
	\draw (\sedge + \dashedge + \edge + \nedge, -\vnedge - \dashedvedge) node (E-m1+2-up) [left] {$-2$};
	\draw (\sedge + \dashedge + \edge + \nedge, -\vnedge - \dashedvedge) node (E-m1+2-down) [right] {$E_{m_1 + 2}$};
	\draw (\sedge + \dashedge + \edge + \nedge, -\vnedge - \dashedvedge -\vedge) node (E-1-up) [left] {$-(m_1+2)$};
	\draw (\sedge + \dashedge + \edge + \nedge, -\vnedge - \dashedvedge -\vedge)  node (E-1-down) [right] {$E_{1}$};

	\end{scope}
\end{tikzpicture}
}
\caption{Dual graphs for the case $r_3 > 0$} \label{fig:euclid-3-2}
\end{center}
\end{figure}

\begin{cor}[{\cite{remmert-van-de-ven}}]\label{rem-de-ven}
Up to an (analytic) isomorphism $\pp^2$ is the only smooth primitive compactification of $\cc^2$.
\end{cor}

\begin{proof}
This follows from combining the first assertions of Theorem \ref{priminimal-resolution-graph} and Proposition \ref{primitive-resolution-graph}.
\end{proof}

\begin{rem}
In \cite{remmert-van-de-ven} Remmert and Van de Ven essentially proved that compactifications of $\cc^2$ are algebraic spaces, i.e.\ Theorem \ref{moishezon-theorem} (which is essentially the point of departure of this article), and then used it to prove the result of Corollary \ref{rem-de-ven} by arguments different from ours. Our proof of Corollary \ref{rem-de-ven} therefore is in fact a new proof of the implication ``Theorem \ref{moishezon-theorem} $\im$ Corollary \ref{rem-de-ven}.''
\end{rem}


\begin{cor}[{\cite{jung}}] \label{jung}
The group of $\cc$-algebra automorphisms of the ring of complex polynomials in two variables is generated by linear automorphisms and {\em triangular} automorphisms (i.e.\ Type II automorphisms of Lemma \ref{normal-lemma}).
\end{cor}

\begin{proof}
Let $\Delta$ be the group of $\cc$-algebra automorphisms of $\cc[x',y']$ and $\Sigma$ be the subgroup of $\Delta$ generated by linear and triangular automorphisms. Pick $F  = (F_1,F_2) \in \Delta$. Set $(u,v) := (F_1(x',y'),F_2(x',y'))$. Let $\bar X \cong \pp^2$ be the compactification of $X := \spec \cc[x',y'] \cong \cc^2$ via the embedding $(x',y') \mapsto [1:F_1(x',y'):F_2(x',y')]$.
Lemma \ref{normal-lemma} implies that there exists $G = (G_1,G_2) \in \Delta$ such that initial exponent of the generic descending Puiseux series $\psi(x,\xi)$ of $C$ with respect to $(x,y) := (G_1(x',y'), G_2(x',y'))$ coordinates is in the normal form. Since $\bar X$ is non-singular, first assertions of Theorem \ref{priminimal-resolution-graph} and Proposition \ref{primitive-resolution-graph} imply that $\psi(x,\xi) = \xi x$ and $G \circ F^{-1}:\bar X \to \Xxy$ is an isomorphism. It follows that $G \circ F^{-1}$ is a non-invertible linear map in $(x,y)$-coordinates; in particular, $G \circ F^{-1} \in \Sigma$. Therefore $F \in \Sigma$, as required. 
\end{proof}

The following result is immediate from the arguments of the proof of Theorem \ref{priminimal-resolution-graph}. We will use it in the proof of Theorem \ref{singular-theorem}. 

\begin{cor} \label{primitive-corollary}
Let $\bar X$ be a primitive compactification of $X$. Choose coordinates $(x,y)$ on $X$ such that the initial exponent of the generic descending Puiseux series $\psi(x,\xi)$ of the curve $C$ at infinity is in the normal form. Then one of the following must hold:
\begin{enumerate}
\item $\bar X \cong \pp^2$, $C \cong \pp^1$.
\item $\sigmaxy^{-1}: \Xxy \dashrightarrow \bar X$ contracts $L_\infty$ to a point $P \in C$. In this case
\begin{enumerate}
\item either $P$ is a singular point of $\bar X$,
\item or $\bar X$ is isomorphic to the weighted projective space $\pp^2(1,1,n)$ for some $n \geq 1$, and $P$ is a non-singular point of both $\bar X$ and $C$.  \qed
\end{enumerate}
\end{enumerate}
\end{cor}

\section{Singularities and curves at infinity} \label{singular-section}

\begin{defn}
Let $\bar X$ be a compactification of $\cc^2$. We say that $\bar X$ is a {\em minimal compactification} (of $\cc^2$) if none of the curves at infinity can be (analytically) contracted.
\end{defn}

\begin{example} \label{minimally-non-singular}
Note that a minimal compactification of $\cc^2$ may {\em not} be a minimal surface. Indeed, let $\bar X_0 = \pp^2$. Pick a line $L_0$ on $\bar X_0$ and a point $Q \in L_0$. Fix $k \geq 1$. Choose points $P_1, \ldots, P_k \in \bar X_0 \setminus L_0$ such that the lines $L_i$ joining $P_i$ and $Q$, $1 \leq i \leq k$, are pairwise distinct. Let $\bar X_k$ be the blow up of $\bar X_0$ at $P_1, \ldots, P_k$. Let $C_i \subseteq \bar X_k$ be the strict transform of $L_i$, $0 \leq i \leq k$, and $X_k := \bar X_k \setminus \bigcup_{i=0}^k C_i$. Note that $X_0 \cong \cc^2$. It then follows by induction that $X_k \cong \cc^2$ (indeed, we need only the following observation: if $Y$ is the blow up of $\cc^2$ at a point $P$ and $L$ is a line through $P$ on $\cc^2$, then the complement in $Y$ of the strict transform of $L$ is also isomorphic to $\cc^2$). We claim that $\bar X_k$ is a minimal compactification of $X_k$. Indeed, the matrix of intersection numbers $(C_i, C_j)$ is:
\begin{align*}
\scrI = \begin{pmatrix}
1 & 1 & 1 & 1 & \cdots & 1 & 1 \\
1 & 0 & 1 & 1 & \cdots & 1 & 1 \\
1 & 1 & 0 & 1 & \cdots & 1 & 1 \\
\vdots &\vdots & \vdots & \vdots & \cdots & \vdots & \vdots \\
1 & 1 & 1 & 1 & \cdots & 0 & 1 \\
1 & 1 & 1 & 1 & \cdots & 1 & 0
\end{pmatrix}
\end{align*}
It then follows from Theorem \ref{greorem} that no $C_i$ can be contracted, i.e.\ $\bar X_k$ is a minimal compactification of $X_k$. Also note that the configuration of the curves at infinity is as in figure \ref{fig:minimal-configuration}. 
\end{example}

\begin{proof}[Proof of Theorem \ref{singular-theorem}]
Let $\sigma: \bar X \dashrightarrow \bar X_0 \cong \pp^2$ be the bimeromorphic correspondence induced by identification of $X$ with $\cc^2$. Algorithm \ref{alg-resolution} shows that a resolution of singularities $\tilde \sigma: \tilde X \to \bar X$ can be constructed from $\bar X_{0}$ via a sequence of blow-ups $\bar X_{i+1} \to \bar X_i$, with $\tilde X = \bar X_s$ for some $s \geq 1$. Let $E_0$ be the line at infinity on $\bar X_0$, and for each $i \geq 1$, let $E_i$ be the exceptional divisor of the $i$-th blow-up. Finally, for each $i \geq 0$, let $\Gamma_i$ be the {\em augmented dual graph} at the $i$-th step, i.e.\ $\Gamma_i$ is the dual graph for the union of strict transforms on $\bar X_i$ of $E_0, \ldots, E_i$. Then it is straightforward to see (see e.g.\ \cite[Remark 5.5]{spivaluations}) that for each $i$, there are only two possibilities for the transformation from $\Gamma_i$ to $\Gamma_{i+1}$ which are described in figure \ref{fig:dual-change}. In particular, it follows that for all $i \geq 1$,
\begin{enumerate}
\renewcommand{\theenumi}{\roman{enumi}}
\item \label{gamma-i-connected} $\Gamma_i$ is a {\em tree} (i.e.\ every pair of vertices is connected by a unique {\em minimal} path). 
\item \label{two-intersection} $E_i$ is connected to at most two distinct $E_j$'s in $\Gamma_i$; denote them as $E_{t_i}$ and $E_{t'_i}$ ($t_i = t'_i$ in the case of figure \ref{fig:dual-change-1}). (Here we used $i > 0$.) 
\item \label{position-0} Let $\tilde \Gamma_i$ (resp.\ $\tilde \Gamma'_i$) be the connected component of $\Gamma_s \setminus \{E_i\}$ which contains $E_{t_i}$ (resp.\ $E_{t'_i}$). Then the vertex corresponding to $E_0$ is in $\tilde \Gamma_i \cup \tilde \Gamma'_i$ (from now we will abuse the notation and simply write $E_0\in \tilde \Gamma_i \cup \tilde \Gamma'_i$). W.l.o.g.\ we assume that $E_0 \in \tilde \Gamma_i$.
\item \label{possible-contraction} Let $\tilde \Gamma$ be a connected component of $\Gamma_s \setminus \{E_i\}$ which does not contain $E_{t_i}$ or $E_{t'_i}$ and let $\tilde E := \bigcup\{E_j:E_j \in \tilde \Gamma\}$. Then there is a point $Q \in E_i$ such that $\tilde E$ is precisely the union of exceptional curves arising from blow-up of $Q$ and points infinitely close to $Q$. In particular, $\tilde E$ can be (analytically) contracted to a non-singular point and the image of $E_i$ under this contraction is also non-singular.  
\end{enumerate}

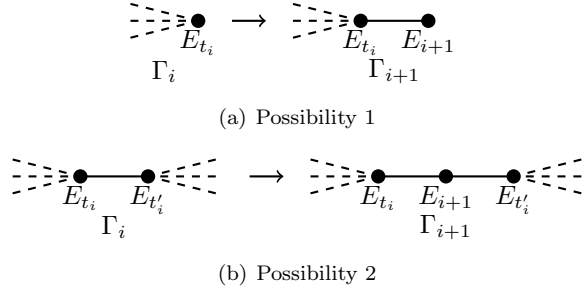
\begin{figure}[htp]
\subfigure[Possibility 1]{
\label{fig:dual-change-1}
\begin{tikzpicture}[scale=.9]
 	\pgfmathsetmacro\edge{1}
 	\pgfmathsetmacro\dashedge{1}
 	\pgfmathsetmacro\dashedheight{0.25}
 	
 	\pgfmathsetmacro\arrowleftsep{0.5}
 	\pgfmathsetmacro\arrowrightsep{0.4}
 	\pgfmathsetmacro\arrowlen{0.5}
 	
 	\pgfmathsetmacro\orleftsep{1}
 	\pgfmathsetmacro\orightsep{1}
 	
 	\pgfmathsetmacro\labelheight{-0.75}

 	\fill[black] (0, 0) circle (3pt);
 	\draw (0,0)  node (E-i-1-1) [below] {$E_{t_i}$};
 	\draw (-\dashedge/2,\labelheight)  node  {$\Gamma_i$};
 	
 	\draw[dashed, thick]  (-\dashedge,\dashedheight) -- (0,0);
 	\draw[dashed, thick]  (-\dashedge,0) -- (0,0);
 	\draw[dashed, thick]  (-\dashedge,-\dashedheight) -- (0,0);
 	
 	\draw[thick, ->] (\arrowleftsep,0) -- (\arrowleftsep + \arrowlen,0);
 	
 	\begin{scope}[shift={(\arrowleftsep + \arrowlen + \arrowrightsep + \dashedge,0)}]
 		
 		\fill[black] (0, 0) circle (3pt);
 		\draw (0,0)  node (E-i-1-2) [below] {$E_{t_i}$};
 	
 		\draw[dashed, thick]  (-\dashedge,\dashedheight) -- (0,0);
 		\draw[dashed, thick]  (-\dashedge,0) -- (0,0);
 		\draw[dashed, thick]  (-\dashedge,-\dashedheight) -- (0,0);
 		
 		\draw[thick] (0,0) -- (\edge,0);
 		
 		\fill[black] (\edge, 0) circle (3pt);
 		\draw (\edge,0)  node (E-i+1-1) [below] {$E_{i+1}$};
 		\draw (\edge/2,\labelheight)  node  {$\Gamma_{i+1}$};
 	
 	\end{scope}
 	\end{tikzpicture}
 }
 
\subfigure[Possibility 2]{
\label{fig:dual-change-2}
\begin{tikzpicture}[scale=0.9]
 	\pgfmathsetmacro\edge{1}
 	\pgfmathsetmacro\dashedge{1}
 	\pgfmathsetmacro\dashedheight{0.25}
 	
 	\pgfmathsetmacro\arrowleftsep{0.5}
 	\pgfmathsetmacro\arrowrightsep{0.4}
 	\pgfmathsetmacro\arrowlen{0.5}
 	
 	\pgfmathsetmacro\orleftsep{1}
 	\pgfmathsetmacro\orightsep{1}
 	
 	\pgfmathsetmacro\labelheight{-0.75} 	
 		
 		\fill[black] (0, 0) circle (3pt);
 		\draw (0,0)  node (E-i-2-1) [below] {$E_{t_i}$};
 	
 		\draw[dashed, thick]  (-\dashedge,\dashedheight) -- (0,0);
 		\draw[dashed, thick]  (-\dashedge,0) -- (0,0);
 		\draw[dashed, thick]  (-\dashedge,-\dashedheight) -- (0,0);
 		
 		\draw[thick] (0,0) -- (\edge,0);
 		\draw (\edge/2,\labelheight)  node  {$\Gamma_i$};
 		
 		\fill[black] (\edge, 0) circle (3pt);
 		\draw (\edge,0)  node (E-j-1) [below] {$E_{t'_i}$};
 		
 		\draw[dashed, thick]  (\edge + \dashedge,\dashedheight) -- (\edge,0);
 		\draw[dashed, thick]  (\edge + \dashedge,0) -- (\edge,0);
 		\draw[dashed, thick]  (\edge + \dashedge,-\dashedheight) -- (\edge,0);

 		\draw[thick, ->] (\edge + \dashedge + \arrowleftsep,0) -- (\edge + \dashedge + \arrowleftsep + \arrowlen,0);
 	
 	\begin{scope}[shift={(\arrowleftsep + \arrowlen + \arrowrightsep + 2*\dashedge + \edge,0)}]
 		
 		\fill[black] (0, 0) circle (3pt);
 		\draw (0,0)  node (E-i-2-1) [below] {$E_{t_i}$};
 	
 		\draw[dashed, thick]  (-\dashedge,\dashedheight) -- (0,0);
 		\draw[dashed, thick]  (-\dashedge,0) -- (0,0);
 		\draw[dashed, thick]  (-\dashedge,-\dashedheight) -- (0,0);
 		
 		\draw[thick] (0,0) -- (\edge,0);
 		
 		\fill[black] (\edge, 0) circle (3pt);
 		\draw (\edge,0)  node (E-i+1-2) [below] {$E_{i+1}$};
 		\draw (\edge,\labelheight)  node  {$\Gamma_{i+1}$};
 		
 		\draw[thick] (\edge,0) -- (2*\edge,0);
 		
 		\fill[black] (2*\edge, 0) circle (3pt);
 		\draw (2*\edge,0)  node (E-j-2) [below] {$E_{t'_i}$};
 		
 		\draw[dashed, thick]  (2*\edge + \dashedge,\dashedheight) -- (2*\edge,0);
 		\draw[dashed, thick]  (2*\edge + \dashedge,0) -- (2*\edge,0);
 		\draw[dashed, thick]  (2*\edge + \dashedge,-\dashedheight) -- (2*\edge,0);
 		 	
 	\end{scope} 	
 	
\end{tikzpicture}
}
\caption{Change of the augmented dual graph in $(i+1)$-th step} \label{fig:dual-change}
\end{figure}

\paragraph{\bf Proof of assertion \eqref{non-minimal-bound}:} $\bar X$ is constructed from $\tilde X$ by contracting some of the $E_i$'s. Let $E_{i_1}, \ldots, E_{i_k}$ be the non-contracted curves; w.l.o.g. we may assume $C_j = \tilde \sigma(E_{i_j})$, $ 1 \leq j \leq k$. Observation \eqref{possible-contraction} can be reformulated as:
\begin{enumerate}
\setcounter{enumi}{4}
\renewcommand{\theenumi}{\roman{enumi}}
\item \label{possible-contraction-2} Fix $j$, $1 \leq j \leq k$. Let $P \in \sing(\bar X) \cap C_j$. Then one of the following folds:
\begin{enumerate}
\item \label{singular-2} $P \in C_j \cap C_{j'}$ for some $j' \neq j$, 
\item \label{singular-0} $i_j \geq 1$, $\tilde \sigma$ contracts $\tilde \Gamma_{i_j} \ni E_0$, and $P = \tilde \sigma(E_0) = \tilde \sigma(\tilde \Gamma_{i_j})$. 
\item \label{singular-not-0} $i_j \geq 1$, $t_{i_j} \neq t'_{i_j}$, $\tilde \sigma$ contracts $\tilde \Gamma'_{i_j}$, and $P = \tilde \sigma(\tilde \Gamma'_{i_j})$. 
\end{enumerate} 
\end{enumerate}
Define
\begin{align}
\Sigma &:= \{i_j:\ \text{$\tilde \sigma$ contracts $\tilde \Gamma_{i_j} \cup \tilde \Gamma'_{i_j}$}\} \label{Sigma}\\
S &= \bigcup_{1 \leq j < j' \leq k}C_j \cap C_{j'} \label{S}
\end{align} 
Observation \eqref{gamma-i-connected} implies that $|S| \leq k-1$. If $\Sigma = \emptyset$, then observation \eqref{possible-contraction-2} implies that for all $j$, $1 \leq j \leq k$, $|\sing(\bar X) \cap C_j \setminus S| \leq 1$. It follows that $|\sing(\bar X)| \leq k + |S| + 1 \leq 2k$. On the other hand, if $\Sigma \neq \emptyset$, then observations \eqref{position-0} and \eqref{gamma-i-connected} imply that $\tilde \sigma$ contracts $E_0$ to some point $P_0 \in \bar X$ and $P_0$ is the unique point of intersection of all $C_j$ such that $i_j \in \Sigma$. Observation \eqref{possible-contraction-2} then implies that for all $j$, $1 \leq j \leq k$, $|\sing(\bar X) \cap C_j \setminus \left(S \cup \{P_0\}\right)| \leq 1$. It follows that $|\sing(\bar X)| \leq k + |S \cup \{P_0\}| \leq 2k$. This completes the proof of assertion \eqref{non-minimal-bound}.\\

\paragraph{\bf Proof of assertion \eqref{minimal-assertion}:} Since $\bar X$ is minimal, it follows from Theorem \ref{greorem} that either $\tilde \sigma$ contracts $E_0$ or $\bar X \cong \bar X_0$. W.l.o.g.\ we may assume the former. Consider the surface $\bar X'$ obtained from $\tilde X$ by contracting all curves at infinity other than the strict transforms of $C_1, \ldots, C_k$ and the line $E_0$ at infinity on $\bar X_0$ (which is possible e.g.\ by Theorem \ref{greorem}). The bimeromorphic correspondences $\pi': \bar X' \dashrightarrow \bar X$ and $\pi_0: \bar X' \dashrightarrow \bar X_{0}$ extend to holomorphic maps. In particular, for each $j$, $1 \leq j \leq k$, the strict transform $C'_j$ of $C_j$ on $\bar X'$ is contractible, so that $(C'_j, C'_j) < 0$. On the other hand, the minimality assumption on $\bar X$ and Theorem \ref{greorem} imply that $(C_j, C_j) \geq 0$ for all $j$, $1 \leq j \leq k$. Since $\pi'|_{\bar X'\setminus E'_0}$ (where $E'_0$ is the strict transform of $E_0$ on $\bar X'$) is an isomorphism, it follows that $E'_0$ intersects each $C'_j$, $1 \leq j \leq k$, so that $P := \pi'(E'_0) \in \bigcap_{j=1}^k C_j$. This, together with observation \eqref{gamma-i-connected} above, implies that $\pi'^{-1}(P) \cap C'_j$ consists of a single point $P'_j$. In particular this proves assertion \eqref{minimal-configuration} and implies that $S \cup \tilde \sigma(E_0) = \{P\}$, where $S$ is as in \eqref{S}. Observation \eqref{possible-contraction-2} then implies that $\left|\sing(\bar X) \setminus \{P\}\right| \leq k$, which is precisely assertion \eqref{minimal-singularities}. Now fix $j$, $1 \leq j \leq k$, and let $\pi'_j: \bar X' \to \bar X^*_j$ be the contraction of all $C'_i$, $i \neq j$. Then $\bar X^*_j$ is precisely the compactification $\bar X^*$ of Remark \ref{easy-2-remark} for $C = C_j$. Since $\pi'_j$ is an isomorphism on a neighborhood of $C'_j \setminus \{P'_j\}$, assertions \eqref{minimal-C-P} and \eqref{minimal-quotient} follows from Remark \ref{easy-2-remark}.\\

\paragraph{\bf Proof of assertion \eqref{non-minimal-types}:} At first note that if $\tilde \sigma$ does {\em not} contract $E_0$, then $\bar X$ dominates $\bar X_0$, and therefore {\em all} singularities of $\bar X$ are sandwiched. So assume $\bar \sigma$ contracts $\tilde E_0$ to a point $P \in \bar X$. Let $\bar X'$ be as in the preceding paragraph. Then $\bar X'$ dominates $\bar X_0$, and therefore all the singularities of $\bar X'$ are sandwiched. Since $\bar X\setminus \{P\} \cong \bar X' \setminus E'_0$, this implies assertion \eqref{non-minimal-types}. \\

\paragraph{\bf Proof of assertion \eqref{curve-singularity}:} At first note that if $C_j$ is the image of $E_0$, then $C_j \cong \pp^1$ (since then the birational map $\bar X \to \bar X_0$ maps $C_j$ on to $L_\infty$). So assume that $E_0$ does {\em not} map on to $C_j$. Let $\bar X^*_j$ be as in the proof of assertion \eqref{minimal-assertion} and $\tilde \pi^*_j: \tilde X \to \bar X^*_j$ be the corresponding map. Recall that $E_{i_j}$ is the strict transform of $C_j$ on $\tilde X$. Let $\tilde Q_j$ be the point of intersection of $E_{i_j}$ and $\tilde \Gamma_{i_j}$ (where $\tilde \Gamma_i$'s are as in observation \eqref{position-0}). Then $\tilde \pi^*_j(\tilde Q_j)$ is precisely the point of intersection of the two curves at infinity on $\bar X^*_j$. Since the bimeromorphic correspondence $\bar X \dashrightarrow \bar X^*_j$ restricts to a holomorphic map on a neighborhood of $C_j\setminus \tilde \sigma(\tilde Q_j)$, assertion \eqref{curve-singularity-1} follows from Remark \ref{easy-2-remark}. \\


It remains to prove assertion \eqref{curve-singularity-2}. Let $Q$ be a singular point of $C_j$ such that $Q \in C_j \setminus \bigcup_{i \neq j}C_i$. Recall that our proof of Theorem \ref{singular-theorem} started with the choice of an {\em arbitrary} compactification $\bar X_0$ of $X$ which is isomorphic to $\pp^2$. Now we choose coordinates $(x,y)$ on $X$ such that the initial exponent of the generic descending Puiseux series associated to $C_j$ is in the {\em normal form}, and set $\bar X_0 = \Xxy$ and $\sigma_0 = \sigmaxy$. The arguments in the preceding paragraph imply that $Q = \tilde \sigma(\tilde Q_j)$ and $\tilde \sigma$ contracts $E_0$ to $Q$. Since $Q \in \bar X \setminus \bigcup_{i\neq j}C_i$, this in turn implies that the bimeromorphic correspondence $\bar X \dashrightarrow \bar X^*_j$ restricts to a holomorphic map on a neighborhood of $E^*_0 := \tilde \pi^*_j(E_0)$. In particular, this implies that $E^*_0$ is analytically contractible. Let $\mu^*_j: \bar X^*_j \to Z$ be the contraction of $E^*_0$. Then $Z$ is a primitive compactification of $X$, and $\mu^*_j$ induces a holomorphic map $\mu_j: \bar X \to Z$ such that $Z\setminus X = \mu_j(C_j)$ and $\mu_j$ is an isomorphism near $Q$. Assertion \eqref{curve-singularity-2} now follows from Corollary \ref{primitive-corollary}. 
\end{proof}

\begin{example}[Compactifications with maximal number of singular points] \label{sharp-example-1}
Pick relatively prime integers $p, q > 1$ and let $\bar X_0$ be the weighted projective surface $\pp^2(1,p,q)$, so that $\bar X_0$ is a compactification of $\cc^2$ with two singular points at infinity. Pick $P \in C := \bar X_0 \setminus X$ such that $\bar X_0$ is non-singular at $P$. Then perform a sequence of $3$ blow-ups as follows: at first blow up $\bar X_0$ at $P$, then blow up the resulting surface at a point on the exceptional divisor $E_1$ which is not on the strict transform of $C$, and then blow up the point of intersection of the new exceptional divisor $E_2$ and the strict transform of $E_1$. This produces a compactification of $\cc^2$ with the dual graph of the union of the curves at infinity as in Figure \ref{fig:sharp-1-configuration}.\\

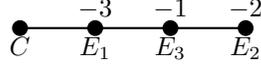
\begin{figure}[htp]
\begin{center}
\begin{tikzpicture}
 	\pgfmathsetmacro\edge{1}
 	
 	\fill[black] (0, 0) circle (3pt);
 	\draw (0,0)  node (C) [below] {$C$};
 	
 	\fill[black] (\edge, 0) circle (3pt);
 	\draw (\edge,0)  node (E1) [below] {$E_1$};
 	\draw (\edge,0)  node (E1-int) [above] {$-3$};
 	
 	\fill[black] (2*\edge, 0) circle (3pt);
 	\draw (2*\edge,0)  node (E3) [below] {$E_3$};
 	\draw (2*\edge,0)  node (E3-int) [above] {$-1$};
 	
 	\fill[black] (3*\edge, 0) circle (3pt);
 	\draw (3*\edge,0)  node (E2) [below] {$E_2$};
 	\draw (3*\edge,0)  node (E2-int) [above] {$-2$};
 	
 	\draw[thick]  (0,0) -- (3*\edge,0);
 	
\end{tikzpicture}
\caption{Construction of $\bar X$ such that $|\sing(\bar X)|$ is maximal} \label{fig:sharp-1-configuration}
\end{center}
\end{figure}
It follows that blowing down $E_1$ and $E_2$ produces a compactification of $\cc^2$ with $2$ irreducible curves and $4$ singular points at infinity. For each $k \geq 1$, applying this procedure to $k$ distinct points on $C \setminus (\sing \bar X_0)$ produces a compactification of $\cc^2$ with $k+1$ irreducible curves and $2(k+1)$ singular points at infinity.
\end{example}

\section{Intersection numbers of curves at infinity} \label{inter-section}

\begin{proof}[Proof of Theorem \ref{intersection-thm}]
Since each $\nu_j$ is centered at infinity, it follows that there exists a compactification $\bar X_j$ of $X$ such that $\nu_j$ is the order of vanishing along a curve $C'_j$ at infinity on $\bar X_j$. By assumption we can assume $\bar X_1 \cong \pp^2$. Let $\tilde X$ be the simultaneous resolution of singularities of $\bar X_j$, $1 \leq j \leq k$. Let $\tilde C_j$ be the strict transform of $C'_j$ on $\tilde X$. Let $\tilde E_1$ be the union of the exceptional curves of the map $\tilde \sigma_1: \tilde X \to \bar X_1$ and let $\tilde E$ be the union of all curves in $\tilde E_1$ which are different from $\tilde C_2, \ldots, \tilde C_k$. Since $\tilde E_1$ is contractible, it follows that $\tilde E$ is also contractible. Let $\tilde \sigma: \tilde X \to \bar X$ be the contraction of $\tilde E$. Then $\bar X$ is precisely the compactification Question \ref{existential-question} asks for. Since $\tilde \sigma_1$ factors through $\tilde \sigma$, it follows that every singularity of $\bar X$ is {\em sandwiched}, and therefore {\em rational} \cite[Remark 1.15]{spinash}. A criterion of Artin \cite[Theorem 2.3]{artractability} then shows that $\bar X$ is projective. This completes the proof of assertions \eqref{easy-existence} and \eqref{easy-rational} of Theorem \ref{intersection-thm}.\\

We now prove assertion \eqref{easy-intersection}. Remark \ref{mremark} shows that $m_{ij} = -\nu_i(g_{\nu_j}(x,y,\tilde \xi))$ for generic $\tilde \xi \in \cc$, where $g_{\nu_j}$ is the {\em generic key form} of $\nu_j$. For all $\tilde \xi \in \cc$, let $D_{1,\tilde \xi}$ be the closure in $\bar X$ of the curve $g_{\nu_1}(x,y,\tilde \xi) = 0$. Recall (from Example \ref{degrexample}) that $D_{1,\tilde \xi}$ is a line with `slope' $\tilde \xi$. Therefore for generic $\tilde \xi$, $D_{1,\tilde \xi}$ intersects $C_0$ transversally at one point and does not intersect any $C_j$ for $j \geq 1$. Since the Weil divisor on $\bar X$ of $g_{\nu_1}(x,y,\tilde \xi)$ is $D_{1,\tilde \xi} + \sum_{l=1}^k \nu_l(g_{\nu_1}(x,y,\tilde \xi))C_l$, it follows that
\begin{align} \label{eqn-no-name-0}
\sum_{l=1}^k m_{l1}(C_l,C_j) = \delta_{1j}\quad \text{for all}\ j,\ 1 \leq j \leq k,
\end{align}
where $\delta_{ij}$ is the usual Kronecker delta. Now fix $i$, $2 \leq i \leq k$ and pick $n \geq 0$ such that $x^n g_{\nu_i} \in \cc[x,y,\xi]$. For all $\tilde \xi \in \cc$, let $D_{i,\tilde \xi}$ be the closure in $\bar X$ of the curve $\{x^n g_{\nu_i}(x,y,\tilde \xi) = 0\} \subseteq X$. Let $Z_{i,\tilde \xi}$ be the image of $D_{i,\tilde \xi}$ under the natural birational morphism $\sigma_1: \bar X \to \bar X_1$. Note that
\begin{enumerate}
\item $\bar X_1 = \Xxy$ and $\sigma_1 = \sigmaxy$ in the notation of Section \ref{curvettection}.
\item $Z_{i,\tilde \xi}$ is precisely the curve $Z_{\tilde \xi}$ from Proposition \ref{curvette-equation} when applied to $\nu = \nu_i$.
\item $\sigma_1^{-1}(Q_y) \in C_1 \setminus\left( \bigcup_{j=2}^k C_k \right)$.
\end{enumerate}
Proposition \ref{curvette-equation} then implies that for generic $\tilde \xi \in \cc$,
\begin{align} \label{eqn-no-name-1}
(D_{i,\tilde \xi}, C_j) 
	&= -n\sum_{l=1}^k \nu_l(x)(C_l,C_j) + \sum_{l=1}^k m_{li}(C_l,C_j) 
	= \begin{cases}
		n 			&\text{if}\ j = 1 \\
		\delta_{ij}	&\text{if}\ 1 < j \leq k.
	\end{cases}
	= n \delta_{1j} + \delta_{ij}
\end{align}
Now recall that by our assumption $\nu_l(x) \leq \nu_l(y)$ for all $l$, $1 \leq l \leq k$. It follows that 
\begin{align*}
m_{l1} = -\nu_l(y - \tilde \xi x)\ \text{(where $\tilde \xi \in \cc$ is generic)} = -\nu_l(x), 
\end{align*}
which, together with identities \eqref{eqn-no-name-0} and \eqref{eqn-no-name-1} imply that 
\begin{align} \label{eqn-no-name-2}
\sum_{l=1}^k m_{li}(C_l,C_j) = \delta_{ij} \quad \text{for all}\ j,\ 1 \leq j \leq k.
\end{align}
The theorem now follows from identities \eqref{eqn-no-name-0} and \eqref{eqn-no-name-2}.
\end{proof}

\begin{example}[Minimal compactifications with maximal number of singular points] \label{sharp-example-2}
We apply Theorem \ref{intersection-thm} to construct, for each $k \geq 1$, minimal compactifications $\bar X_k$ of $X$ with $k$ irreducible curves at infinity and $|\sing(\bar X_k)| = k+1$. Choose relatively prime positive integers $p, q$. For $k = 1$, the weighted projective space $\pp^2(1,p,q)$ satisfies the requirement, provided both $p$ and $q$ are $\geq 2$. So assume $k \geq 2$. Pick distinct complex numbers $\alpha_2, \ldots, \alpha_{k+1}$ and for each $j$, $2 \leq j \leq k+1$, let $\nu_j$ be the divisorial valuation on $\cc(x,y)$ corresponding to generic descending Puiseux series $\tilde \psi_j(x,\xi) := \alpha_j x + \xi x^{-q/p}$; in other words, $\nu_j$ is the negative of the weighted degree on $\cc(x,y)$ with respect to coordinates $(x, y-\alpha_j x)$ such that the weight of $x$ is $p$ and the weight of $y - \alpha_j x$ is $-q$. The key forms of $\nu_j$ are $x,y,y-\alpha_jx$, and the generic key form of $\nu_j$ is
\begin{align*}
g_{\nu_j} = (y - \alpha_j x)^p - \xi x^{-q},\quad 2 \leq j \leq k+1.
\end{align*}
Let $\nu_1 = -\deg$ and $\bar X$ be the surface obtained by applying Theorem \ref{intersection-thm} to $\nu_1, \ldots, \nu_{k+1}$. Since $g_{\nu_1} = y - \xi x$, it follows that 
\begin{align*}
\scrM &= \begin{pmatrix}
1 & p & p & \cdots & p & p \\
p & -pq & p^2 & \cdots & p^2 & p^2 \\
p & p^2 & -pq & \cdots & p^2 & p^2 \\
\vdots & \vdots & \vdots & \cdots & \vdots & \vdots \\
p & p^2 & p^2 & \cdots & p^2 & -pq
\end{pmatrix}\\ 
\scrI &= \scrM^{-1} = \begin{pmatrix}
1-\frac{kp}{p+q}	& \frac{1}{p+q} 	& \frac{1}{p+q} & \cdots & \frac{1}{p+q} & \frac{1}{p+q} \\
\frac{1}{p+q} 		& -\frac{1}{p(p+q)} & 0 			& \cdots & 0 			 & 0 \\
\vdots 				& \vdots 			& \vdots 		& \cdots & \vdots 		 & \vdots \\
\frac{1}{p+q} 		& 0					& 0				& \cdots & 0			 & -\frac{1}{p(p+q)}
\end{pmatrix}
\end{align*}
Now assume $(k-1)p > q$. Then $(C_1, C_1) < 0$ and therefore $C_1$ is analytically contractible (Theorem \ref{greorem}); let $\bar X_{p,q}$ be the surface formed from $\bar X$ via contracting $C_1$. We claim that for a suitable choice of parameters $p$ and $q$, $\bar X_{p,q}$ is a minimal compactification of $X$ and $|\sing(\bar X_{p,q})| = k + 1$. Let $C'_j$ be the image of $C_j$ on $\bar X_{p,q}$ via the morphism $\pi': \bar X \to \bar X_{p,q}$. For the minimality of $\bar X_{p,q}$ it suffices to show that $(C'_j, C'_j) \geq 0$ for each $j$, $2 \leq j \leq k+1$. But $(C'_j, C'_j) = (\pi'^*(C'_j), \pi'^*(C'_j)) = (C_j + c_jC_1, C_j + c_jC_1)$, where $c_j = -(C_1, C_j)/(C_1,C_1)$. Consequently, 
\begin{align*}
(C'_j, C'_j) = (C_j + c_jC_1, C_j) = \frac{(C_1,C_1)(C_j,C_j) - (C_1,C_j)^2}{(C_1,C_1)} = \frac{1}{p(p+q)}\frac{q - (k-2)p}{(k-1)p - q}
\end{align*}
Since $(k-1)p > q$, it follows that $(C'_j, C'_j) \geq 0$ iff $(k-2)p \leq q$, i.e.\ $\bar X_{p,q}$ is indeed a minimal compactification of $X$ if $k \geq 2$ and $(k-2)p \leq q < (k-1)p$. \\

Now we compute $|\sing(\bar X_{p,q})|$. First note that for $2 \leq i < j \leq k+1$,
\begin{align*}
(C'_i, C'_j) = (C_i + c_iC_1, C_j + c_jC_1) = (C_i + c_iC_1, C_j) = c_i(C_1,C_j) = \frac{1}{(p+q)((k-1)p - q)};
\end{align*}
in particular, $(C'_i,C'_j)$ is {\em not} an integer, which implies that the (unique) point $P'$ of intersection of $C'_i$ and $C'_j$ (which is also the point of intersection of {\em all} $C'_l$, $2 \leq l \leq k+1$, due to assertion \ref{minimal-configuration} of Theorem \ref{singular-theorem}) is singular. To see other singular points of $\bar X_{p,q}$, note that for each $j$, $1 \leq j \leq k$, there is a morphism $\pi_j: \bar X \to \bar X_{p,q,j}$, where $\bar X_{p,q,j}$ is the surface obtained from $\bar X$ by contracting all curves at infinity other than $C_1$ and $C_j$. Since $-\nu_1$ and $-\nu_j$ are weighted degrees in $(x, y - \alpha_j x)$-coordinates, it follows that $\bar X_{p,q,j}$ is the toric surface corresponding to the polygon of figure \ref{fig:bar-X-p-q-j}. It follows from basic toric geometry that if $p \geq 2$, then $\bar X_{p,q,j}$ has a singular point $Q_j$ on $\pi_j(C_j) \setminus \pi_j(C_1)$. Since $\pi_j$ is invertible near $P_j$, it then follows that $P_j := \pi_j^{-1}(Q_j)$ is a singular point on $C_j \setminus C_1$ and consequently the image $P'_j$ of $P_j$ on $\bar X_{p,q}$ is a singular point on $C'_j \setminus \bigcup_{i\neq j}C_i$. 

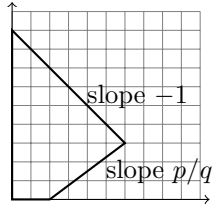
\begin{figure}[htp]
\begin{center}
\begin{tikzpicture}[scale=0.25]
	\begin{scope}[shift={(0,0)}]
		\draw [gray,  line width=0pt] (0,0) grid (10,10);
		\draw [<->] (0,10.5) |- (10.5,0);
		
		\coordinate (x) at (2,0);
		\coordinate (z) at (6,3);
		\coordinate (y) at (0,9);
				
		\draw[thick] (0,0) -- (x) -- (z) -- (y) -- cycle;		
		
		\draw (4.5,1.5) node (slope1) [right, text width=2cm]{slope $p/q$};
		\draw (3.5,5.5) node (slope2) [right, text width=2cm]{slope $-1$};

	\end{scope}
\end{tikzpicture}
\caption{Polygon corresponding to $\bar X_{p,q,j}$}  \label{fig:bar-X-p-q-j}
\end{center}
\end{figure} 
\end{example}

\begin{proof}[Proof of Theorem \ref{determinanthm}]
W.l.o.g.\ we may (and will) assume that no two $\nu_j$'s are mutually proportional. We divide the proof in two cases:\\

{\bf Case 1: there exists $j$, $1 \leq j \leq k$, such that $\nu_j = - \deg$.} In this case w.l.o.g.\ we may assume $j = 1$ and Theorem \ref{intersection-thm} shows that the answer is affirmative. So we only have to show that $\det(- M) < 0$. Indeed, let $\scrI$ be the intersection matrix of the curves at infinity on $\bar X$ and $\tilde \scrI$ be the $(k-1) \times (k-1)$ submatrix of $\scrI$ with $(i,j)$-th entry being $(C_i, C_j)$, $2 \leq j \leq k$. Since $C_2 \cup \cdots \cup C_k$ is contractible, Grauert's theorem (Theorem \ref{greorem}) implies that $\tilde \scrI$ is negative definite. Similarly, since $C_1 \cup \cdots \cup C_k$ is {\em not} contractible, it follows that $\scrI$ is not negative definite. Since $\det(\scrI) \neq 0$, it then follows from the standard test of negative-definiteness via the sign of principal minors that $(-1)^{k}\det \scrI < 0$. Consequently, $(-1)^{k}\det \scrM = \det(-\scrM) < 0$, as required. \\

{\bf Case 2: there is no $j$, $1 \leq j \leq k$, such that $\nu_j = - \deg$.} In this case, let $\nu_0 = -\deg$ and apply Theorem \ref{intersection-thm} to the collection $\nu_0, \ldots, \nu_k$. Let $\bar X'$ be the resulting compactification of $\cc^2$ and $\scrI'$ be the matrix of intersection numbers of curves $(C'_i,C'_j)$, where $C'_i$ is the curve at infinity on $\bar X'$ corresponding to $\nu_i$, $0 \leq i \leq k$. Theorem \ref{intersection-thm} implies that $\det \scrM$ is precisely the $(1,1)$-minor of $\scrM' := \scrI'^{-1}$. Cramer's rule then implies that $(C'_0,C'_0) = \det \scrM/\det \scrM'$. On the other hand, applying Case 1 to $\nu_0, \ldots, \nu_k$ yields that $\sign(\det \scrM') = (-1)^k$. Consequently, $\sign((C'_0,C'_0)) = \sign((-1)^k\det \scrM) = \sign(\det(-\scrM))$. Now the result follows from Grauert's theorem.
\end{proof}

As an application of Theorem \ref{determinanthm}, we give an interpretation of {\em skewness} of valuations - an invariant of valuations defined by Favre and Jonsson in order to study the {\em valuative tree} (see \cite{favsson-tree} for details). 

\begin{defn}[see {\cite[Appendix A]{favsson-eigen}}]
Let $\nu$ be a divisorial discrete valuation on $\cc[x,y]$ centered at infinity such that $\nu(x) < 0$ and $\nu(x) \leq \nu(y)$. Assume that $\nu \neq -\deg$, where $\deg$ is the degree in $(x,y)$-coordinates. Let $P$ be the center of $\nu$ on $\Xxy \cong \pp^2$. For every $f \in \sheaf_{\Xxy, P}$, let $\tilde m(f)$ be the intersection multiplicity at $P$ of the curve $\{f = 0\}$ with the line at infinity. Note that $u:= 1/x$ is a regular function at $P$ and $u = 0$ is precisely the equation of the line at infinity near $P$. Let $\tilde \nu := \nu/\nu(u)$ be the {\em normalized version} of $\nu$ (in the sense that $\tilde \nu(u) = 1$). Then the {\em relative skewness} of $\nu$ is $\tilde \alpha (\nu) := \sup\{\tilde \nu(f)/\tilde m(f): f \in \sheaf_{\Xxy, P}\}$ and the {\em skewness} of $\nu$ is $\alpha(\nu) := 1 - \tilde \alpha(\nu)$.\footnote{In \cite[Appendix A]{favsson-eigen} skewness was defined only for normalized valuations centered at infinity. We simply defined the skewness of a valuation centered at infinity to be the skewness of its normalized version.}
\end{defn}

\begin{cor} \label{cor-primitive-existence}
Let $\nu$ be a divisorial discrete valuation on $\cc(x,y)$ centered at infinity such that $\nu(x) < 0$ and $\nu(x) \leq \nu(y)$. Let $g_\nu$ be the generic key form of $\nu$ with respect to $(x,y)$-coordinates. Then the following are equivalent:
\begin{compactenum}
\item \label{primitive-valuation} $\nu$ determines a compactification of $X$ (i.e.\ there is a (unique) compactification $\bar X$ of $X$ such that the curve $C$ at infinity on $\bar X$ is irreducible and $\nu$ is the order of vanishing along $C$).
\item \label{generically-negative} $\nu(g_\nu(x,y,\tilde \xi)) < 0$ for some (and hence every!) $\tilde \xi \in \cc$. 
\item \label{positively-skewed} $\alpha(\nu) > 0$.
\end{compactenum}
\end{cor}

\begin{proof}
Let $p := \deg_y(g_\nu)$. Recall (from Definition \ref{key-definition}) that $g_\nu = \tilde g_\nu/u^p = \frac{1}{u^p}(\tilde g_l^{n_l} - \xi \prod_{j=0}^{l-1} \tilde g_j^{n_j})$, where $\tilde g_j$'s are key polynomials of $\nu$ with respect to $(u,v) := (1/x,y/x)$-coordinates. The defining properties of key polynomials then imply that  
\begin{align}
&p = n_l\deg_v(\tilde g_l) = \nu(u),\ \text{and for all $\tilde \xi \in \cc$},\\
&\nu(g_\nu(x,y,\tilde \xi)) 
	= n_l\nu(\tilde g_l) - \nu(u^p) 
	= n_l\nu(\tilde g_l) - p\nu(u)
	= n_l(\nu(\tilde g_l) - \nu(u)\deg_v(\tilde g_l)).
\end{align}
In particular, $\nu(g_\nu(x,y,\tilde \xi))$ does not depend on $\tilde \xi$. The equivalence of assertions \ref{primitive-valuation} and \ref{generically-negative} then immediately follows from the $k=1$ case of Theorem \ref{determinanthm}. On the other hand, \cite[Lemma 3.32]{favsson-tree} implies that 
\begin{align*}
\tilde \alpha(\nu) = \frac{\tilde \nu(\tilde g_l)}{\tilde m(\tilde g_l)} = \frac{\nu(\tilde g_l)}{\nu(u)\deg_v(\tilde g_l)}
\end{align*}
It follows that 
\begin{align*} 
\alpha(\nu) 
	&= 1 - \tilde \alpha(\nu) 
	= \frac{\nu(u)\deg_v(\tilde g_l) - \nu(\tilde g_l)}{\nu(u)\deg_v(\tilde g_l)}
	= - \frac{\nu(g_\nu(x,y,\tilde \xi))}{p\nu(u)}
	= - \frac{\nu(g_\nu(x,y,\tilde \xi))}{p^2}
\end{align*}	
which shows the equivalence of assertions \ref{generically-negative} and \ref{positively-skewed}, and completes the proof of the corollary.
\end{proof}

\begin{rem} 
The term $\nu(g_\nu(x,y,\tilde \xi))$ from assertion \eqref{generically-negative} of Corollary \ref{cor-primitive-existence}, or equivalently the skewness $\alpha(\nu)$ can be calculated in a straightforward way in terms of formal Puiseux pairs of the generic descending Puiseux series $\psi_\nu(x,\xi)$ of $\nu$. We present the formula for the sake of completion: let $(q_1, p_1), \ldots, (q_{l+1}, p_{l+1})$ be the formal Puiseux pairs of $\psi_\nu$. Set $p := p_1 \cdots p_{l+1}$. Then for every $\tilde \xi \in \cc$,
\begin{align*} 	
\nu(g_\nu(x,y,\tilde \xi)) 	
	&= -p \left( (p_1 \cdots p_{l+1} - p_2 \cdots p_{l+1}) \frac{q_1}{p_1} + (p_2 \cdots p_{l+1} - p_3 \cdots p_{l+1}) \frac{q_2}{p_1p_2} \right. \\
	& \quad \qquad \left. + \cdots + (p_l p_{l+1} - p_{l+1}) \frac{q_l}{p_1 \cdots p_l} + p_{l+1} \frac{q_{l+1}}{p_1 \cdots p_{l+1}} \right)
\end{align*}	
\end{rem}

\bibliographystyle{alpha}
\bibliography{bibi}
%

\end{document}